\documentclass[a4paper,12pt]{amsart}

\usepackage{geometry}
\usepackage[comma,numbers]{natbib}
\usepackage[inline]{enumitem}
\usepackage{graphicx}
\usepackage[textfont=it,tableposition=above,font=small]{caption}		

\usepackage{framed}					
\usepackage[foot]{amsaddr}	
\usepackage{url}						

\usepackage{xargs}					
\usepackage{amssymb}				

\usepackage{silence}
\usepackage{todonotes}

\usepackage[ruled,vlined]{algorithm2e}

\usepackage{multirow}										
\usepackage{arydshln}										
\usepackage{psfrag}
\usepackage{mathtools}									

\usepackage{color}
\usepackage{xcolor}
\RequirePackage[colorlinks,citecolor=blue,urlcolor=blue]{hyperref}
\usepackage[norefs,nocites]{refcheck}

\usepackage{randtext}

\newtheorem{theorem}{Theorem}
\newtheorem{lemma}[theorem]{Lemma}
\newtheorem{corollary}[theorem]{Corollary}

\theoremstyle{definition}
\newtheorem*{definition}{Definition}
\newtheorem{example}{Example}
\newtheorem*{remark}{Remark}

\newcommand{\onefig}{.55}                             
\newcommand{\twofig}{.45}                             

\newcommand{\Meas}[1][\R^d]{\mathcal M\left(#1\right)}
\newcommand{\D}{D}																		

\newcommand{\PP}{\mathsf P}														
\newcommand{\R}{\mathbb R}														
\newcommand{\e}{\mathrm e}														
\newcommand{\tr}{^\mathsf{T}}													
\newcommand{\Sph}[1][d]{\mathbb{S}^{#1-1}}						

\newcommand{\Haus}{\delta_H}													

\DeclareMathOperator{\intrOp}{int}									  
\DeclareMathOperator{\clOp}{cl}   									  
\DeclareMathOperator{\bdOp}{bd}												
\DeclareMathOperator{\relintOp}{relint}								
\DeclareMathOperator{\relbdOp}{relbd} 								
\DeclareMathOperator{\relclOp}{relcl}									

\newcommand{\medianOp}{D^*}														
\newcommand{\coverMOp}{C^*}									          

\DeclareMathOperator{\facesOp}{\mathcal{F}}             

\DeclareMathOperator{\affOp}{aff}	                    
\DeclareMathOperator{\convOp}{conv}			          		
\DeclareMathOperator{\vol}{vol}   									  
\DeclareMathOperator{\area}{\lambda}									

\newcommand{\intr}[1]{\intrOp\left(#1\right)}
\newcommand{\cl}[1]{\clOp\left(#1\right)}
\newcommand{\bd}[1]{\bdOp\left(#1\right)}
\newcommand{\relint}[1]{\relintOp\left(#1\right)}
\newcommand{\relcl}[1]{\relclOp\left(#1\right)}
\newcommand{\relbd}[1]{\relbdOp\left(#1\right)}

\newcommand{\faces}[1]{\facesOp\left(#1\right)}

\newcommand{\median}[1][\mu]{\medianOp\left(#1\right)}
\newcommand{\coverM}[1][\mu]{\coverMOp\left(#1\right)}

\newcommand{\Damu}{\D_{\alpha}(\mu)}                  
\newcommand{\Uamu}{U_{\alpha}(\mu)}										
\newcommand{\UaFBmu}{U_{\alpha}^{FB}(\mu)}						
\newcommand{\Ucamu}{U^\circ_{\alpha}(\mu)}            

\newcommand{\aff}[1]{\affOp\left(#1\right)}
\newcommand{\conv}[1]{\convOp\left(#1\right)}

\newcommand{\compl}{^\mathsf{c}}                      

\newcommand{\half}{\mathcal H}												

\title[Ray basis theorem]{Halfspace depth for general measures: The ray basis theorem and its consequences}

\author{Petra Laketa}
\author{Stanislav Nagy}
\email{\randomize{nagy@karlin.mff.cuni.cz}}

\address{
	Charles University,
	Faculty of Mathematics and Physics,
	Prague, Czech Republic
}
	
\date{\today}

\begin{document}

\begin{abstract}
The halfspace depth is a prominent tool of nonparametric multivariate analysis. The upper level sets of the depth, termed the trimmed regions of a measure, serve as a natural generalization of the quantiles and inter-quantile regions to higher-dimensional spaces. The smallest non-empty trimmed region, coined the half{\-}space median of a measure, generalizes the median. We focus on the (inverse) ray basis theorem for the halfspace depth, a crucial theoretical result that characterizes the halfspace median by a covering property. First, a novel elementary proof of that statement is provided, under minimal assumptions on the underlying measure. The proof applies not only to the median, but also to other trimmed regions. Motivated by the technical development of the amended ray basis theorem, we specify connections between the trimmed regions, floating bodies, and additional equi-affine convex sets related to the depth. As a consequence, minimal conditions for the strict monotonicity of the depth are obtained. 
Applications to the computation of the depth and robust estimation are outlined.
\end{abstract}

\maketitle

%
%

\section{Introduction: Characterization of the halfspace median}

The lack of natural ordering of the points in multidimensional spaces makes the nonparametric analysis of multivariate data challenging. The depth introduces a data-dependent ordering of the sample points, in the direction from the most central observations being those that attain high depth values, to the peripheral ones with low depth \cite{Liu_etal1999, Zuo_Serfling2000, Zuo_Serfling2000b, Chernozhukov_etal2017}. A definition of depth-based central regions of the data, which are the regions where the depth exceeds given thresholds, ensues naturally. The smallest non-empty depth region is often termed the depth median set of the data. The depth medians provide convenient robust location estimators, well studied in the literature.

We consider the seminal halfspace (or Tukey) depth, and the general setup of finite Borel (not necessarily probability) measures. Write $\R^d$ for the $d$-dimensional Euclidean space equipped with the inner product $\left\langle \cdot, \cdot \right\rangle$, and $\Sph$ for the unit sphere in $\R^d$. The set of all closed halfspaces in $\R^d$ is denoted by $\half$. Elements of $\half$ can be represented as 
	\[	H_{x,u} = \left\{ y \in \R^d \colon \left\langle x, u \right\rangle \leq \left\langle y, u \right\rangle \right\}	\]
for $x \in \R^d$ a point in the boundary hyperplane and $u\in\Sph$ the inner unit normal of $H_{x,u}$. We write $\half(x) = \left\{ H_{x,u} \colon u \in \Sph \right\}$ for all halfspaces whose boundary passes through $x \in \R^d$. The collection of all finite Borel measures on $\R^d$ is denoted by $\Meas$. The halfspace depth was first considered by \citet{Tukey1975} and later substantially elaborated on by \citet{Donoho_Gasko1992}. The \emph{halfspace depth} of a point $x\in \R^d$ with respect to (w.r.t.) a measure $\mu\in \Meas$ is defined as
	\begin{equation}	\label{halfspace depth}
	\D(x;\mu)=\inf_{H\in \half(x)}\mu(H) = \inf_{H \in \half, x \in H} \mu(H).
	\end{equation}
The two expressions for the depth are easily seen to be equivalent. Our principal interest lies in the concept of the \emph{depth-trimmed regions} (also called \emph{central regions}) of $\mu$, defined for $\alpha \in \R$ by $\Damu = \left\{ x \in \R^d \colon \D\left(x;\mu\right) \geq \alpha \right\}$. It is a simple observation that the regions $\Damu$ are always closed convex sets, non-increasing in $\alpha \in \R$ in the sense of set inclusion. For $\alpha \leq 0$ we have $\Damu = \R^d$; for all $\alpha$ large enough $\Damu = \emptyset$. The supremum of all $\alpha$ such that $\Damu \neq \emptyset$ is denoted by $\alpha^*(\mu)$. We call the set $D_{\alpha^*}(\mu)$ the \emph{(halfspace) median set} of $\mu$, and denote it by $\median$. The median set is always non-empty and compact. Elements of $\median$ are called \emph{(halfspace) medians} of $\mu$. Any halfspace $H \in \half(x)$ that satisfies $\D(x;\mu) = \mu(H)$ is called a \emph{minimizing halfspace} of $\mu$ at $x \in \R^d$.

The so-called ray basis theorem provides a convenient characterization of a halfspace depth median of a measure $\mu$ in terms of its minimizing halfspaces. In its simplest form, the theorem asserts that for measures $\mu \in \Meas$ with continuous and positive density, a point $x \in \R^d$ is a half{\-}space median of $\mu$ if and only if the collection of its minimizing halfspaces covers the whole space $\R^d$. That result was first proved by \citet[Claim on p.~1818]{Donoho_Gasko1992}, and later extended in \citet[Propositions~8 and~12]{Rousseeuw_Ruts1999}. For the special case of uniform measures on convex bodies, such a characterization of the deepest point relates to an early observation of \citet[p.~251]{Grunbaum1963} from convex geometry. For details and additional discussion about the ray basis theorem, its history and relevance in both statistics and geometry see \cite[Section~4.3.1]{Nagy_etal2019} and \cite{Patakova_etal2020}. 

Our initial goal is to revisit the ray basis theorem, and consolidate its statement by extending it to any depth region $\Damu$ of a general measure $\mu \in \Meas$, under minimal assumptions.\footnote{The extension of these results from Borel probability measures to finite Borel measures on $\R^d$ is minor. By our treatment of \emph{general} measures we mean generalizations of the ray basis theorem and related results to measures from $\Meas$ that do not have to possess positive densities, or satisfy other simplifying conditions concerning, e.g., their support.} We do so in Section~\ref{section:ray basis}, where tools from convex geometry are employed to devise an elementary proof of a generalized version of the theorem. We provide conditions under which it is possible to cover the complement to $\Damu$ by halfspaces of $\mu$-mass $\alpha$. As a special case, we obtain connections between the median set, and the set of all points that allow covering $\R^d$ by their minimizing halfspaces. An important part of our contribution are the examples, that throughout the paper demonstrate that the conditions stated in our main results cannot be avoided. 

In Section~\ref{section:level sets} we thoroughly discuss the links of our general ray basis theorem with the properties of the central regions $\Damu$. It is known from the literature \cite[Proposition~6]{Rousseeuw_Ruts1999} that, writing $A\compl$ for the complement to $A \subseteq \R^d$ in $\R^d$, 
	\begin{equation} \label{intersection of closed halfspaces}
	\Damu = \bigcap \left\{H \in \half \colon \mu(H\compl)< \alpha\right\}.
	\end{equation}
Each $\Damu$ is thus an intersection of closed convex sets, and must be closed and convex itself. We begin from \eqref{intersection of closed halfspaces}, and specify relations between the central regions $\Damu$, the trimmed regions as considered in \cite{Nolan1992, Masse_Theodorescu1994} and \cite[Section~3.9.4.6]{VanDerVaart_Wellner1996}, and the floating body known from convex geometry \cite{Bobkov2010, Nagy_etal2019}. 

Our paper is concluded with three applications of our results in Section~\ref{section:applications}. The set of depth medians $\median$ is not necessarily a single point set. Especially for empirical measures $\mu$, that is measures corresponding to datasets, the median set is frequently full-dimensional. In Section~\ref{section:covering median} we propose to single out the collection of those medians that satisfy an additional covering property. We obtain a smaller collection of covering medians of $\mu$, which share qualitatively better properties than the general elements of $\median$. An algorithm for finding covering medians is given. In Section~\ref{section:computation} we obtain a consequence of the general ray basis theorem regarding the structure of the central regions $\Damu$ for $\mu$ an empirical measure. In that case, each facet $F$ of the convex polytope $\Damu$ is shown to lie in a hyperplane determined by data points. This observation promises applications in the computation of the trimmed regions $\Damu$ also in the case when they are not full-dimensional, in a spirit similar to that lately employed in \cite{Liu_etal2019}. Finally, in Section~\ref{section:Dyckerhoff} we provide, as an interesting by-product of our study, the minimal set of assumptions that guarantees the depth to be strictly monotone. As argued by \citet[Example~4.2]{Dyckerhoff2017}, strict monotonicity is one of the most important properties a depth can have. For a depth w.r.t. a probability measure $\mu$, it ensures the almost sure uniform convergence of the depth upper level sets estimated from the data towards their population counterparts. As such, sufficient conditions for strict monotonicity find applications in the estimation of the depth-trimmed regions $\Damu$ from data in multivariate statistics. The proofs of all theoretical results are provided in an extensive appendix accompanying the paper.

\subsection{Preliminaries and notation}

In the proofs of our results, we use tools from measure theory, as well as from the theory of convex sets. Our general reference to the concepts used from convexity theory is \cite{Schneider2014}. We now set the most important notations used in the paper, and state a preliminary observation about the halfspace depth that will be useful later. For a set $A \subseteq \R^d$ its interior, closure and boundary are denoted by $\intr{A}$, $\cl{A}$ and $\bd{A}$, respectively. Denote by $\aff{A}$ and $\conv{A}$ the affine hull and the convex hull of $A$. The sets $\relint{A}$, $\relcl{A}$ and $\relbd{A}$ represent the (relative) interior, closure and boundary of $A$ in the space $\aff{A}$, and $\dim A=\dim\left(\aff{A}\right)$ is the dimension of $A$. The complement to $A$ is $A\compl=\R^d\setminus A$. For sets $A$ and $B$ we write $A \subset B$ if $A \subseteq B$ and $A \ne B$. We say that $H\in \half$ is a touching halfspace of a non-empty convex set $A$ if $H \cap \cl{A} \neq \emptyset$ and $\intr{H}\cap A = \emptyset$. The collection of all touching halfspaces to $A$ is denoted by $\half(A)$. We also define $\half(\emptyset) = \half$. In this notation, $\half(x)$ is the same as $\half(\{x\})$ for $x \in \R^d$.

\smallskip\noindent\textbf{Well-behaved measures.} We say that $\mu\in\Meas$ is \emph{smooth} if $\mu(\bd{H})=0$ for each $H\in\half$. We call $\mu$ \emph{smooth at} a convex set $A \subset \R^d$ if $\mu(\bd{H}) = 0$ for all $H \in \half(A)$, and smooth at a point $x \in \R^d$ if it is smooth at $\left\{x\right\}$. Smoothness of $\mu$ at a point is a condition stronger than $\mu$ being atom-less; yet, it is still weaker than smoothness in the whole $\R^d$. A measure $\mu$ is said to have \emph{contiguous support} if the support of $\mu$ cannot be separated by a slab between two parallel hyperplanes of non-empty interior with zero $\mu$-mass. Finally, $\mu$ is said to have \emph{contiguous support at} a convex set $A \subset \R^d$ if
	\begin{equation}	\label{contiguous support}
	\mbox{for each $H^\prime \in \half(A)$ and $H^\prime \subset H \in \half$, $\bd{H} \cap \intr{A} \ne \emptyset$ implies $\mu(H^\prime) < \mu(H)$.}	
	\end{equation}
Note that the condition \eqref{contiguous support} is void if $\intr{A} = \emptyset$, and therefore it is enough to consider $A$ full-dimensional. In that case, \eqref{contiguous support} means that any shift $H \supset H^\prime$ of a touching halfspace $H^\prime$ of $A$ has $\mu$-mass larger than $H^\prime$. It is satisfied if, for instance, the set $A$ is a subset of the support of $\mu$. An absolutely continuous measure is smooth (at any $A$ convex). A measure with connected support has contiguous support (at any convex subset $A$ of the support of $\mu$).

\smallskip\noindent\textbf{Minimizing halfspaces.} For $x \in \R^d$ recall that $H \in \half(x)$ is a minimizing halfspace of $\mu\in\Meas$ at $x$ if $\D(x;\mu) = \mu(H)$. In general, the set of minimizing halfspaces of $\mu$ at a point may be empty. It is guaranteed to be non-empty if, for instance, the measure $\mu$ is smooth at $x$. Our first observation that will prove to be useful in the sequel is that it is always possible to find $H \in \half(x)$ with the property $\mu(\intr{H}) \leq \D\left(x;\mu\right)$. We call such a halfspace $H$ a \emph{generalized minimizing halfspace} of $\mu$ at $x$.

\begin{lemma}	\label{lemma:minimizing halfspace} 
For any $\mu \in \Meas$ there exists a generalized minimizing halfspace of $\mu$ at any point $x\in\R^d$. If $\mu$ is smooth at $x$, then there exists a minimizing halfspace of $\mu$ at $x$.
\end{lemma}

\smallskip\noindent\textbf{Additional notations.} In addition to the upper level set of the halfspace depth $\Damu = \left\{ x \in \R^d \colon \D\left(x;\mu\right) \geq \alpha \right\}$, we also consider the set
	\[	\Uamu = \left\{x\in \R^d\colon \D(x;\mu)>\alpha \right\} = \bigcup_{\beta > \alpha} \D_{\beta}(\mu).	\]
It is convex, but not always closed. As will be seen later in the paper, the set $\Uamu$ is also not open in general. Both $\Damu$ and $\Uamu$ are bounded for $\alpha > 0$.

A face of a convex set $A$ is a convex subset $F \subseteq A$ such that $x, y \in A$ and $(x + y)/2 \in F$ implies $x, y \in F$. Denote by $\faces{A}$ the set of all faces of $\cl{A}$. A facet of $A$ is a face of $A$ of dimension $\dim(A) - 1$. For $F\in\faces{A}$ and $A$ convex, denote by $\half(A,F)=\{H\in\half(A)\colon F\subseteq \bd{H}\}$ the collection of all halfspaces touching $A$ at its face $F$. We say that a sequence of halfspaces $\left\{H_{x_n,u_n}\right\}_{n=1}^\infty \subset \half$ converges to $H_{x,u} \in \half$ if $x_n \to x$ in $\R^d$ and $u_n \to u$ in $\Sph$. 

We write $L(x,y)$ and $L[x,y]$ for the relatively open (that is, not containing $x$ and $y$) and the relatively closed (containing $x$ and $y$) line segment between distinct point $x,y\in \R^d$, respectively, and $l(x,y)$ for the infinite line determined by $x$ and $y$.

%
%

\section{The general ray basis theorem}	\label{section:ray basis}

The standard ray basis theorem, as dubbed by \citet{Rousseeuw_Ruts1999}, asserts that under certain conditions on a measure $\mu \in \Meas$, the halfspace median $x$ of $\mu$ is characterized by the fact that $\R^d$ can be covered by minimizing halfspaces of $x$. The formal statement of the theorem is given below.

\begin{theorem}[{\cite[Propositions~8 and~12]{Rousseeuw_Ruts1999}}]	\label{theorem:standard ray basis}
For $\mu \in \Meas$, the following holds true:
	\begin{enumerate}[label=(\roman*), ref=(\roman*)]
	\item \label{direct ray basis} If for some $x \in \R^d$ 
		\begin{equation}	\label{ray basis}
		\R^d = \bigcup\left\{ H \in \half(x) \colon \mu(H) = \D(x;\mu) \right\},	
		\end{equation}
	then $x$ is a median of $\mu$, i.e. $\D(x;\mu) = \alpha^*(\mu)$. 
	\item \label{inverse ray basis} If $x$ is a median of $\mu$ and, in addition, $\mu$ is absolutely continuous with a density that is continuous and positive in an open convex set, then \eqref{ray basis} is true for $x$.
	\end{enumerate}
\end{theorem}

The proof of the direct part~\ref{direct ray basis} of Theorem~\ref{theorem:standard ray basis} is simple: since $\R^d$ is covered by halfspaces of $\mu$-mass exactly $\D(x;\mu)$, each $y \in \R^d$ must be contained in such a halfspace $H_y$, and consequently $\D(y;\mu) \leq \mu(H_y) = \D(x;\mu)$, as follows directly from the definition of the halfspace depth \eqref{halfspace depth}. The proof of the so-called \emph{inverse ray basis theorem}, stated in part~\ref{inverse ray basis}, is slightly more involved, and technical. Our aim is to obtain a result similar to Theorem~\ref{theorem:standard ray basis} for any depth region $\Damu$, under minimal conditions. We prove that for well-behaved measures it is possible to cover all points outside $\intr{\Damu}$ with closed halfspaces $H$ whose $\mu$-mass is bounded from above by $\alpha$. For measures that may not be smooth the appropriate formulation of the latter condition turns out to be $\mu(\intr{H})\leq \alpha$. In the following lemma we introduce a condition that will play an important role in the formulation of the extended ray basis theorem. We treat not only the depth-trimmed regions $\Damu$, but also $\Uamu$, as well as their interiors. 

\begin{lemma} \label{lemma:basic condition} 
Let $\mu\in \Meas$, $\alpha \in \R$ and $S\in \{\Damu,\Uamu,\intr{\Damu},\intr{\Uamu}\}$. Then
	\begin{equation}	\label{basic condition}
	S \compl \subseteq \left(\intr{S}\right)\compl \subseteq \bigcup	\{	H \in \half \colon \mu(\intr{H})\leq \alpha\mbox{ and }\intr{H}\cap S=\emptyset	\}.
	\end{equation}
\end{lemma}

Application of Lemma~\ref{lemma:basic condition} to the median set $S=\intr{\median}$ gives a slightly weaker version of the standard inverse ray basis theorem. Suppose that the measure $\mu\in\Meas$ satisfies the conditions of part~\ref{inverse ray basis} of Theorem~\ref{theorem:standard ray basis}. First, it is not difficult to observe that for such $\mu$ our auxiliary Lemma~\ref{lemma:minimizing halfspace} about the existence of a generalized minimizing halfspace ensures that the median set $\median$ cannot be full-dimensional.\footnote{This claim is proved under weaker conditions as Corollary~\ref{corollary:dimension of median} in Section~\ref{section:proof of dimension} in the Appendix.} 
Because the median set is always convex, it is not full-dimensional if and only if its interior is empty. Therefore, for any $\mu$ with a positive continuous density, Lemma~\ref{lemma:basic condition} gives that the whole space can be covered by halfspaces \begin{enumerate*}[label=(\roman*)] \item of $\mu$-mass at most $\alpha^*(\mu)$; and \item without a particular connection to any fixed point $x \in \median$. \end{enumerate*} 

Our intention is now to extend the results of Lemma~\ref{lemma:basic condition} by restricting the halfspaces on the right hand side of \eqref{basic condition} to only those that touch $S$, i.e. to the collection $\half(S)$. For that purpose, we introduce additional notation.

\begin{definition} 
For a convex set $A \subset \R^d$ and a point $x\notin \intr{A}$ define 
	\begin{equation*}	
	\mathfrak{F}(x,A) = 
					\begin{cases}
					\left\{F\in\faces{A} \colon \relint{\conv{F\cup \{x\}}} \cap \cl{A} = \emptyset \right\} & \mbox{if } x\notin \cl{A}, \\
					\left\{F\in\faces{A} \colon x \in \relint{F} \right\} & \mbox{if } x \in \bd{A}.
					\end{cases}	
	\end{equation*}
\end{definition}

The collection $\mathfrak{F}(x,A)$ consists of those faces of the closed convex set $\cl{A}$ that are completely visible from a point $x \notin \intr{A}$. Note that $F \in \mathfrak{F}(x,A)$ implies $\dim(F) < d$, and that $\mathfrak{F}(x,\emptyset) = \{\emptyset\}$. The concept of visible faces relates to the theory of illumination of convex bodies by external light sources \cite{Bezdek_Khan2018}; for its application to the statistics of the depth see \cite{Nagy_Dvorak2021}. Note that any point $x \notin \intr{A}$ illuminates at least one non-empty face of a non-empty convex set $A$; for details see Lemma~\ref{lemma:faces} presented in the Appendix.

Now we are able to extend Lemma~\ref{lemma:basic condition} to the touching halfspaces. Starting from \eqref{basic condition}, our intention is to find $H \in \half(S)$ that covers $x \notin S$ given, with the properties as on the right hand side of \eqref{basic condition}. The main idea is to approach any point $y\in F \in \mathfrak{F}\left(x,S\right)$ by a sequence $\left\{y_n\right\}_{n=1}^\infty \subset \left(\cl{S}\right)\compl$, that converges to $y$ from the outside of $S$. Lemma~\ref{lemma:basic condition} gives that for each $y_n$ there is a halfspace $H_n \ni y_n$ whose interior has empty intersection with $S$ and $\mu(\intr{H_n}) \leq \alpha$, see the left hand panel of Figure~\ref{figure:proof}. Since $y_n \to y$, there exists a convergent subsequence of $\{H_n\}_{n=1}^\infty$ whose limiting halfspace $H$ can be shown to be touching $S$, containing $x$, and satisfying $\mu(\intr{H})\leq \alpha$, as we wanted to show. The formal proof of Lemma~\ref{lemma:basic lemma} is postponed to the Appendix.

\begin{figure}[htpb]
\includegraphics[width=\twofig\textwidth]{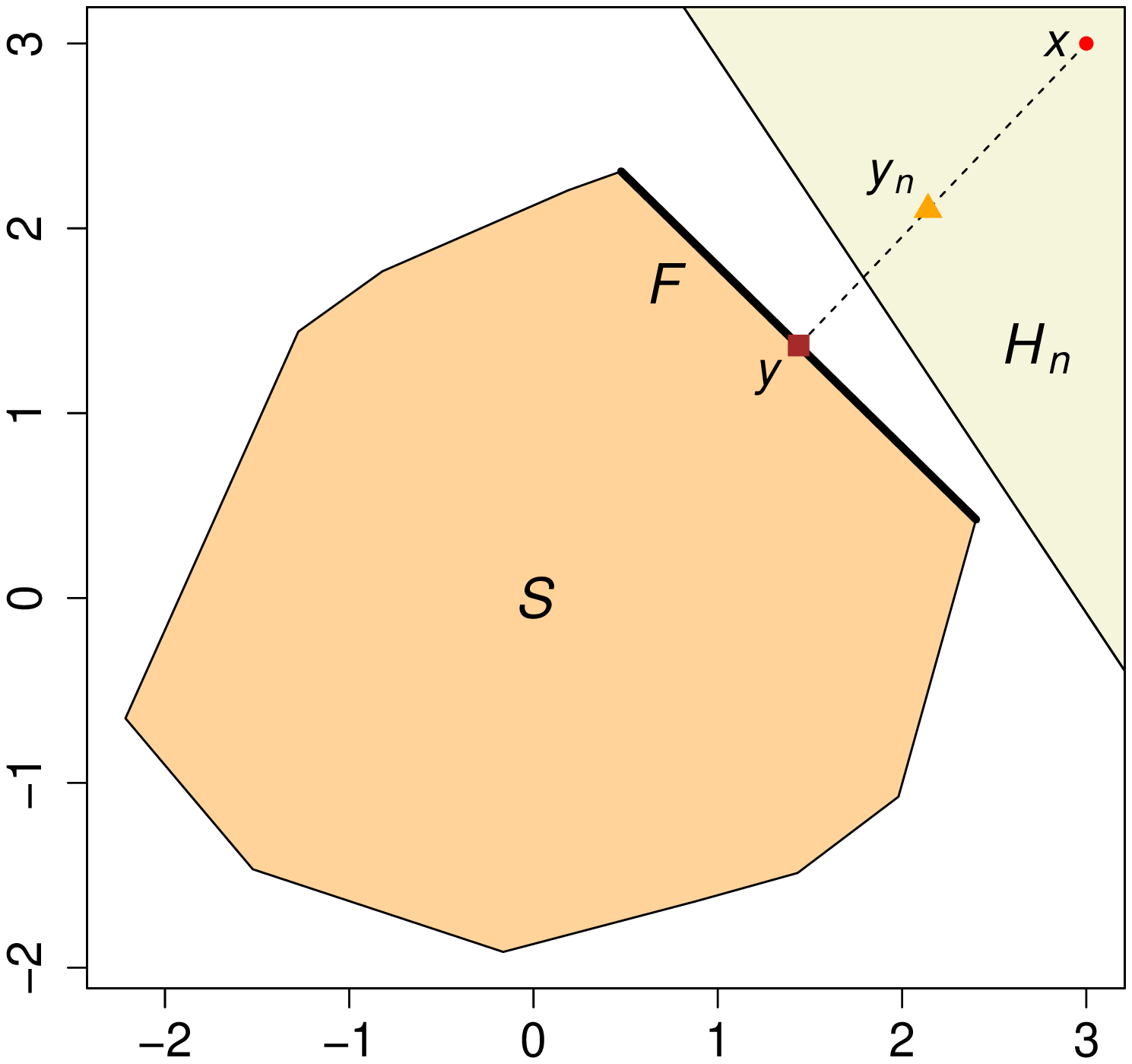} \quad
\includegraphics[width=\twofig\textwidth]{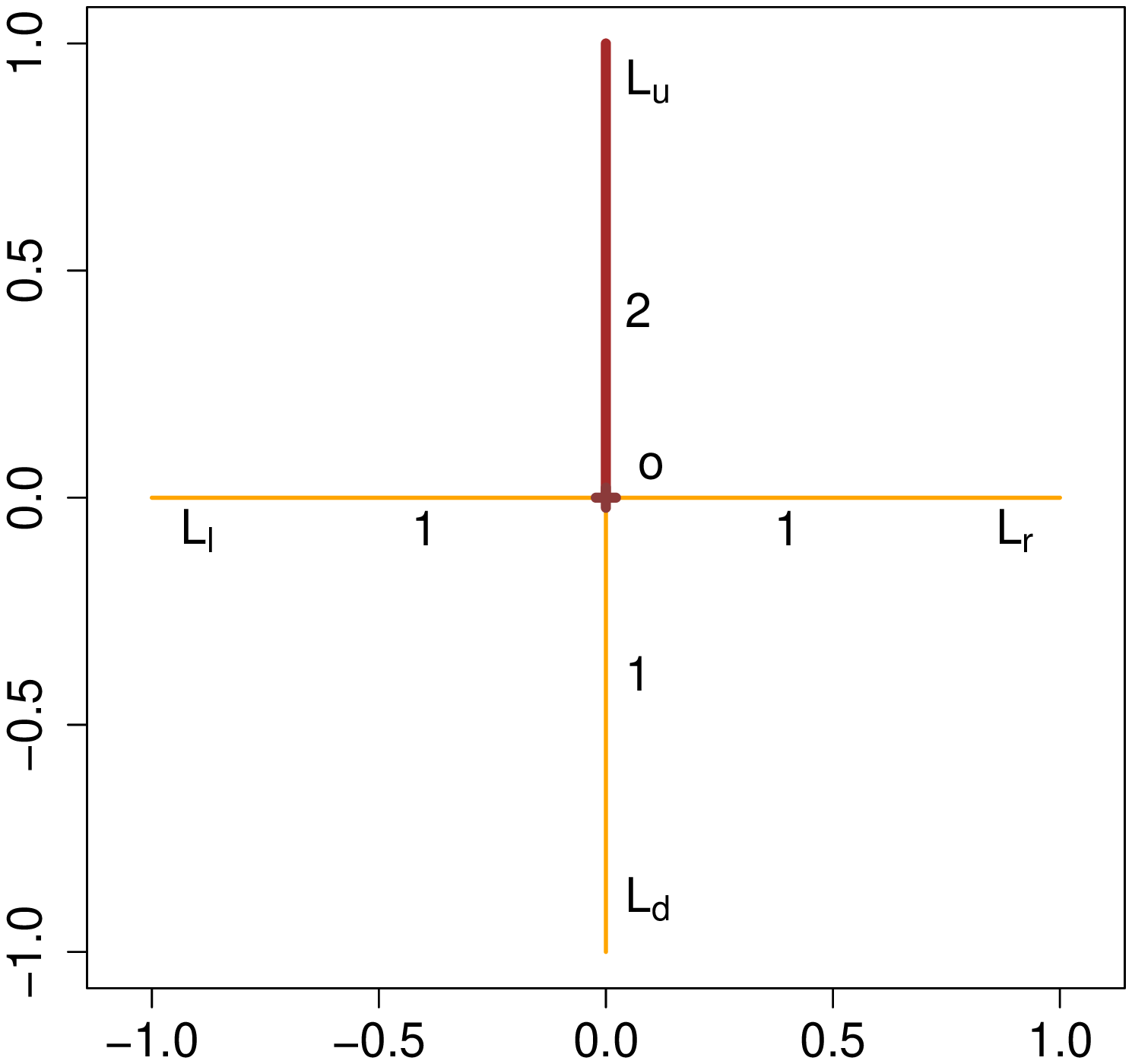}
\caption{Left panel: Proof of Lemma~\ref{lemma:basic lemma}. As $y_n \to y$, a subsequence of halfspaces $\left\{H_n\right\}_{n=1}^\infty$ converges to $H \in \half(y)$ that touches $S$, contains $x$ and the face $F$, and $\mu(\intr{H}) \leq \alpha$. Right panel: Local smoothness is a necessary condition for the inverse ray basis theorem. The origin $o\in\R^2$ is the unique median of $\mu$ from Example~\ref{example:unique median is not enough}, but it fails to satisfy the covering condition \eqref{ray basis}.}
\label{figure:proof}
\end{figure}

\begin{lemma} \label{lemma:basic lemma}
Consider $\mu\in\Meas$ and a bounded convex set $S \subset \R^d$ satisfying \eqref{basic condition} with $\alpha \in \R$. For any $x\notin S$ and $F\in \mathfrak{F}(x,S)$  there exists $H(x,F)\in\half(S,F)$ such that $x\in H(x,F)$ and $\mu(\intr{H(x,F)})\leq \alpha$. In particular,
	\[	 S\compl \subseteq (\intr{S})\compl = \bigcup\left\{ H \in \half(S) \colon \mu(\intr{H}) \leq \alpha \right\}.	\]
\end{lemma}

Lemma~\ref{lemma:basic lemma} is technical, but presents an important intermediate result. It provides multiple consequences that we explore in the sequel. Lemma~\ref{lemma:basic condition} allows us to apply Lemma~\ref{lemma:basic lemma} to the sets $\Damu$, $\Uamu$ or their interiors, for any $\mu \in \Meas$. This is the foundation for the most general statement of the ray basis theorem that can be devised for general measures. To obtain finer results, it is necessary to impose additional assumptions. In the literature on the halfspace depth, two typical assumptions are the smoothness of $\mu$, and the contiguity of its support. We require these conditions only locally, at the considered depth regions. Application of Lemma~\ref{lemma:basic lemma} to $S = \Damu$ yields the following generalization of the inverse ray basis theorem.

\begin{theorem}[\textbf{General inverse ray basis theorem}] \label{theorem:ray basis depth regions}
For any $\mu\in\Meas$, $\alpha\in\R$, $x \notin \intr{\Damu}$ and $F\in \mathfrak{F}(x,\Damu)$ there exists $H(x,F)\in\half(\Damu,F)$ such that $x\in H(x,F)$ and $\mu(\intr{H(x,F)})\leq \alpha$.
In particular,
	\begin{equation}	\label{D equality}
	\left(\intr{\Damu}\right)\compl = \bigcup\left\{ H \in \half(\Damu) \colon \mu(\intr{H}) \leq \alpha \right\}.	
	\end{equation}
Additionally, if	
\begin{enumerate}[label=(\roman*), ref=(\roman*)]
\item	\label{smooth case}	$\mu$ is smooth at $\Damu$, then $\mu(H(x,F))= \alpha$ and 
\[	 \left(\intr{\Damu}\right)\compl = \bigcup\left\{ H \in \half(\Damu) \colon \mu(H) = \alpha \right\}.	\]
\item	\label{contiguous support case}	$\mu$ has contiguous support at $\Damu$, then
\[	\left(\intr{\Damu}\right)\compl = \bigcup\left\{ H \in \half \colon \mu(\intr{H}) \leq \alpha \right\}.	\]
\item	\label{smooth and contiguous support case}	$\mu$ is smooth at $\Damu$ and has contiguous support at $\Damu$, then
\[	\left(\intr{\Damu}\right)\compl = \bigcup\left\{ H \in \half \colon \mu(H) = \alpha \right\}.	\]
\end{enumerate}	
\end{theorem}

Note that if $\Damu$ fails to be full-dimensional, the left hand sides in the formulas in Theorem~\ref{theorem:ray basis depth regions} are all $\R^d$, and the theorem therefore gives conditions under which the whole sample space can be covered by halfspaces of limited $\mu$-mass. For $\Damu$ contained in a hyperplane, also the condition from part~\ref{contiguous support case} of Theorem~\ref{theorem:ray basis depth regions} is trivially satisfied.

Before moving to the discussion about the relevance of Theorem~\ref{theorem:ray basis depth regions}, we state an analogous result for $\Uamu$ as another corollary of Lemma~\ref{lemma:basic lemma}. Since obviously $\Uamu \subseteq \Damu$, its general statement is a refinement of the first claim \eqref{D equality} of Theorem~\ref{theorem:ray basis depth regions}.

\begin{theorem} \label{theorem:ray basis additional}
For any $\mu\in\Meas$, $\alpha\in\R$, $x \notin \intr{\Uamu}$ and $F\in \mathfrak{F}(x,\Uamu)$ there exists $H(x,F)\in\half(\Uamu,F)$ such that $x\in H(x,F)$ and $\mu(\intr{H(x,F)})\leq \alpha$.
In particular,
	\begin{equation}	\label{U equality}
	\left(\intr{\Uamu}\right)\compl = \bigcup\left\{ H \in \half \colon \mu(\intr{H}) \leq \alpha \right\}.	
	\end{equation}
Additionally, if $\mu$ is smooth at $\Uamu$, then $\Uamu=\intr{\Uamu}$, so $\Uamu$ is open.
\end{theorem}

Comparison of Theorem~\ref{theorem:ray basis additional} and part~\ref{contiguous support case} of Theorem~\ref{theorem:ray basis depth regions} draws connections between the depth regions $\Damu$ and $\Uamu$ --- for $\mu$ with contiguous support at $\Damu$, we have $\intr{\Damu} = \intr{\Uamu}$. We postpone this discussion into Section~\ref{section:level sets}, where connections between upper level sets of the depth and related convex constructions are explored thoroughly. In that section, also further applications of these observations are found.

If $\mu$ is not smooth, the last statement of Theorem~\ref{theorem:ray basis additional} certainly cannot be claimed. This was observed already in \cite[Lemma~6]{Struyf_Rousseeuw1999} where it was noted that for $\mu$ an atomic measure on points with unit weights in general position, $\Uamu$ equals $\D_{\alpha+1}(\mu)$ for any $\alpha \in \R$, and as such, $\Uamu$ is always a closed set. Our first consequence of the general ray basis theorem is the following observation concerning the dimensionality of the median set. It presents a refinement of \cite[Proposition~3.4]{Small1987}. 

\begin{corollary}	\label{corollary:dimension of median}
Let $\mu \in \Meas$. 
	\begin{enumerate}[label=(\roman*), ref=(\roman*)]
	\item \label{dimension i} If $\mu$ has contiguous support at $\median$, then $\dim\left(\median\right) \ne d$. 
	\item \label{dimension ii} If $\mu$ is smooth at $\median$ and $d>1$, then $\dim\left(\median\right) \ne d-1$.
	\end{enumerate}
\end{corollary}

Without the assumptions of smoothness and contiguous support, the median set may be of any dimension. Consider, for instance $\mu\in\Meas[\R^2]$ that gives mass $1$ to each of the points $\left(-1,-1\right)\tr$, $\left(-1,1\right)\tr$, $\left(1,1\right)\tr$, and mass $2$ to $\left(1,-1\right)\tr$. It is easy to see that $\alpha^*(\mu) = 2$, and $\median = L\left[\left(0,0\right)\tr, \left(1,-1\right)\tr \right]$. For additional discussion on the dimensionality of the median set for empirical measures we refer to \cite{Liu_etal2020}.

We are now ready to reformulate the inverse ray basis theorem from part~\ref{inverse ray basis} of Theorem~\ref{theorem:standard ray basis}, under minimal assumptions.

\begin{corollary}[\textbf{Inverse ray basis theorem for the median}]	\label{corollary:ray basis}
Suppose that $\mu \in \Meas$ is smooth at $\median$. Then
	\[	\left(\intr{\median}\right)\compl = \bigcup\left\{ H \in \half(\median) \colon \mu(H) = \alpha^*(\mu) \right\}.	\]
If, in addition, $\mu$ has contiguous support at $\median$, then the covering condition \eqref{ray basis} holds true for any $x\in\median$. 
\end{corollary}

The assumptions of Corollary~\ref{corollary:ray basis} are weaker than those in Theorem~\ref{theorem:standard ray basis}: \begin{enumerate*}[label=(\roman*)] \item instead of the existence of the density $f$ of $\mu$ we require only local smoothness of $\mu$, and \item instead of the strict positivity and continuity of $f$ in a neighbourhood of the median we need a weaker condition of locally contiguous support at the median set. \end{enumerate*} 

\begin{remark}[Uniqueness of the median]
In Corollary~\ref{corollary:dimension of median} we show that under the assumptions of both contiguous support and smoothness of $\mu$ at $\median$, the median set cannot be of dimension $d$, or $d-1$. In particular, for $d = 1$ and $2$ it follows that the median must be unique. It is tempting to claim that for a smooth measure $\mu$ with convex support (part~\ref{inverse ray basis} of Theorem~\ref{theorem:standard ray basis}), the median set must consist of a single point, compare with \cite[Proposition~7]{Mizera_Volauf2002} and the discussion in \cite[Section~3]{Small1987}. Surprisingly, it turns out that there exist probability distributions with a density that is smooth and positive in a convex set in dimension $d > 2$ that fail to possess a unique halfspace median. The appropriate conditions for the uniqueness of the halfspace median in higher dimensions turn out to be not only smoothness and contiguous support, but, quite surprisingly, also a certain integrability assumption. For a detailed discussion on the problem of the uniqueness of the halfspace median we refer to \cite{Pokorny_etal2021}. For that reason, one has to be careful when interpreting the inverse ray basis theorem in Corollary~\ref{corollary:ray basis}. Suppose that a measure $\mu$ is smooth with contiguous support. Then for any point $x \in \median$ it is true that the sample space $\R^d$ is covered by minimizing halfspaces of $\mu$ at $x$. That, however, does not mean that the median of $\mu$ must be unique. Especially in higher dimension, the non-trivial median set may still lie in the boundary of all the halfspaces from the covering system.
\end{remark}

We continue by giving examples that demonstrate that the assumptions of Corollary~\ref{corollary:ray basis} are difficult to be weakened. In our first example we show that without local smoothness, even in the case when $x$ is the unique median of $\mu$, it may fail to satisfy the covering condition \eqref{ray basis}.

\begin{example}	\label{example:unique median is not enough}
Consider $\mu\in\Meas[\R^2]$ whose support is plotted in the right hand panel of Figure~\ref{figure:proof}. It is given as a mixture of the uniform distributions on line segments $L_u=L\left[\left(0,0\right)\tr, \left(0,1\right)\tr \right]$, $L_d=L\left[\left(0,0\right)\tr, \left(0,-1\right)\tr \right]$,  $L_l=L\left[\left(0,0\right)\tr, \left(-1,0\right)\tr \right]$ and $L_r=L\left[\left(0,0\right)\tr, \left(1,0\right)\tr \right]$, such that $\mu(L_l)=\mu(L_r)=\mu(L_d)=1$ and $\mu(L_u)=2$. The origin $o = \left(0,0\right)\tr$ is the unique median of $\mu$ with $\D(o;\mu)=2$, since for any other point $x\in \R^2$ there is a halfspace $H\in\half(x)$ that is parallel with one of the axes such that $\mu(H)<2$. At the same time, each closed halfspace containing $L_u$ has $\mu$-mass at least $3$, implying that it is impossible to cover (any point from) $L_u$ by $H\in\half(o)$ such that $\mu(H) \leq \alpha^*(\mu) = 2$. On the other hand, observe that part~\ref{contiguous support case} of Theorem~\ref{theorem:ray basis depth regions} is valid, as we easily find open halfspaces whose boundary passes through the origin with mass at most $\alpha^*(\mu)$ whose closures cover $\R^2$.
\end{example}

It is known that for a uniform distribution $\mu$ on a triangle $\Delta$ in the plane, the barycentre $o\in\Delta$ of $\Delta$ is the unique median of $\mu$ \cite[Section~5.3]{Rousseeuw_Ruts1999}. We construct an example of a uniform distribution on the set obtained by removing a narrow strip containing $o$ from $\Delta$. We show that our measure does not satisfy the assumption of contiguous support at $\median$, at the same time its median set is full-dimensional, and contains points that fail to cover the plane by their minimizing halfspaces as in~\eqref{ray basis}. 

\begin{example}	\label{example:triangle gap}
Consider the equilateral triangle in $\R^2$ centred at the origin $o = \left(0,0\right)\tr$ determined by points $a=(0,2)\tr$, $b=(-\sqrt{3},-1)\tr$ and $c=(\sqrt{3},-1)\tr$ as displayed in Figure~\ref{figure:triangle gap}. For $x \in (0,2)$ and $y \in (0,1)$ denote $x_c = \left(0,x\right)\tr$ and $y_c = \left(0,-y\right)\tr$, and let $x_l$ and $x_r$ be the points where the horizontal line containing $x_c$ intersects $L(a,b)$ and $L(a,c)$, respectively. Analogously, define points $y_l$ and $y_r$ for the horizontal line that contains $y_c$. Let $\mu \in \Meas[\R^2]$ be the uniform distribution on the set $S = \conv{\{a, x_l, x_r\}} \cup \conv{ \{ y_l, b, c, y_r \}}$ with total mass being the area of $S$. It is possible to choose $x$ and $y$ positive, but small enough so that the set $\median$ is full-dimensional, while $y_c\in\median$ is the only point in $\R^2$ that satisfies the covering property \eqref{ray basis}. For a detailed technical proof of this claim, we refer to Section~\ref{section:triangle gap} in the Appendix. 

\begin{figure}[htpb]
\includegraphics[width=\twofig\textwidth]{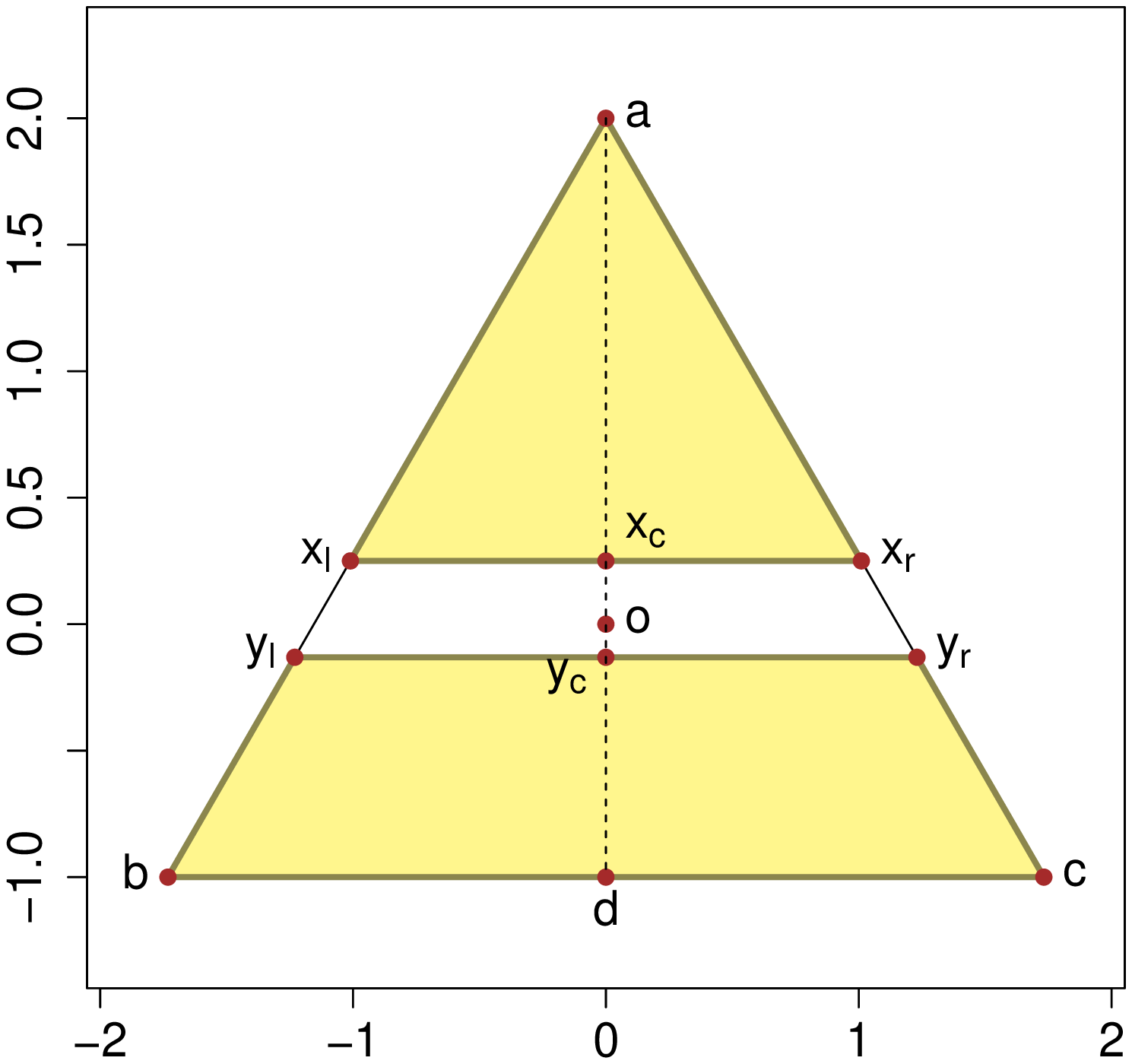}
\includegraphics[width=\twofig\textwidth]{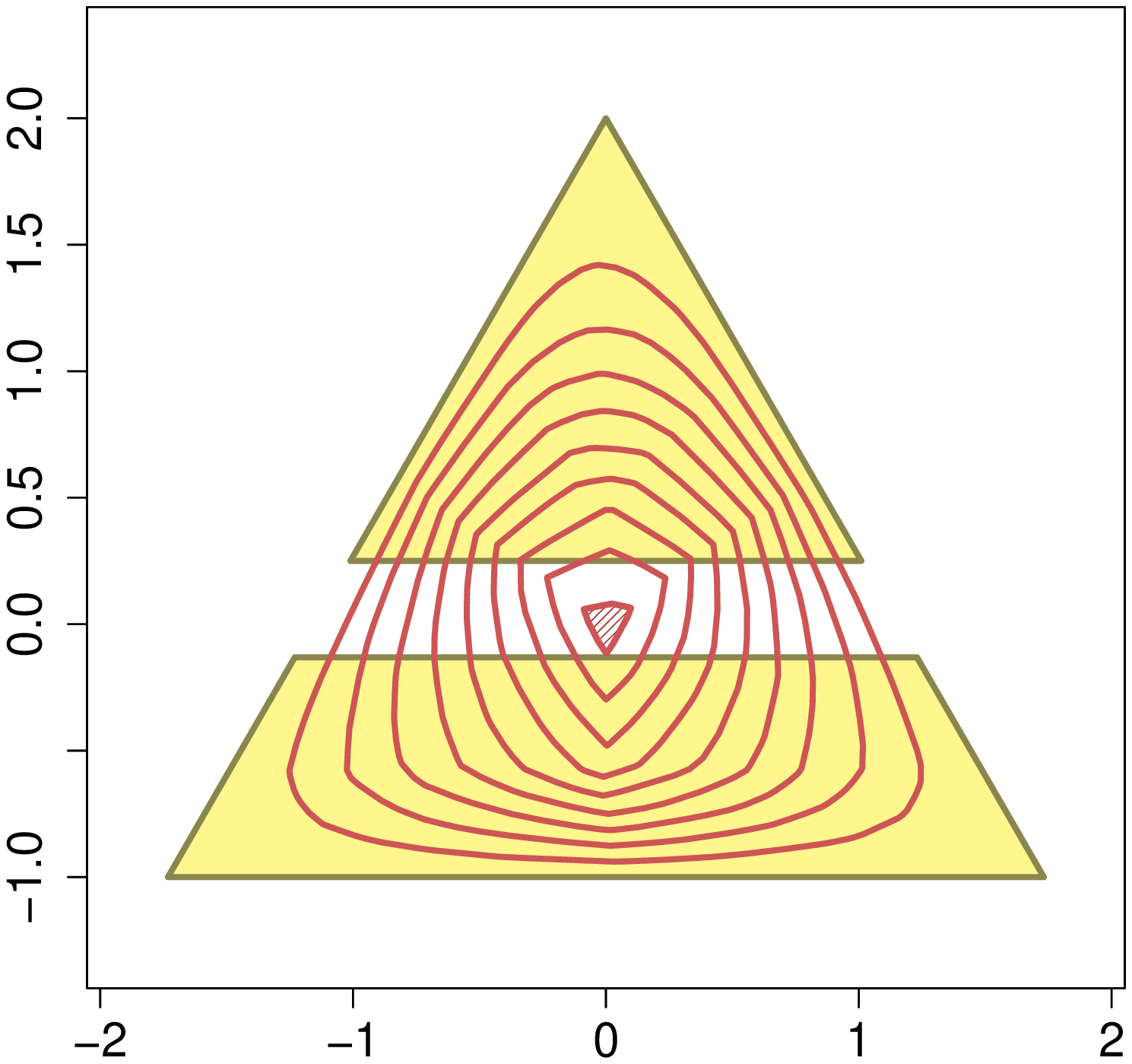}
\caption{Example~\ref{example:triangle gap}: For the inverse ray basis theorem, it is not enough to assume only the smoothness of $\mu$, without the property of contiguous support. For the uniform measure $\mu$ on the coloured region, the median is a full-dimensional set, yet the point $y_c$ displayed in the left hand panel is the only point in $\R^2$ that satisfies the covering condition \eqref{ray basis}. On the right hand panel we see several numerically computed depth regions $\Damu$, with the median set $\median$ being the smallest region, located in the closed strip removed from the triangle.}
\label{figure:triangle gap}
\end{figure}
\end{example}

%
%

\section{Depth regions and floating body}	\label{section:level sets}

Formula \eqref{intersection of closed halfspaces} allows us to write any depth region $\Damu$ as an intersection of closed halfspaces whose complements have $\mu$-mass smaller than $\alpha$. Another important affine equivariant set is the \emph{floating body} of $\mu$ corresponding to $\alpha \in \R$
	\[	\UaFBmu = \bigcap \left\{H \in \half \colon \mu(H\compl)\leq \alpha\right\},	\]
as defined in \cite[Section~5]{Bobkov2010}. According to the discussion in \cite{Nagy_etal2019}, the floating bodies are of great interest in both geometry and probability theory. For the special case of $\mu$ a uniform distribution on a full-dimensional convex set $K \subset \R^3$, the floating body has a compelling mechanical interpretation --- the set $\UaFBmu$ can be described as the part of the convex solid $K$ of (volumetric mass) density $\alpha \in (0,\alpha^*(\mu))$ that never submerges beneath the surface of water of unit density, when fully rotated on the surface. The history of the research on floating bodies goes well into the 19th century. In statistics, a construction equivalent to the floating body of a measure is much more recent, and sometimes referred to as the multivariate trimming \cite{Nolan1992,Masse_Theodorescu1994}. Being intersections of closed sets, both $\Damu$ and $\UaFBmu$ are closed. In what follows we use the results of Section~\ref{section:ray basis}, and precise the connections between the depth regions $\Damu$, $\Uamu$, the floating body $\UaFBmu$, and a further set that turns out to be of interest in our analysis
	\[	\Ucamu = \bigcap \left\{\intr{H}\colon H\in\half,\ \mu(H\compl)\leq \alpha\right\}.	\]
We demonstrate that this last region is an upper level set of a function closely related to the halfspace depth
	\[	\D^\circ(x;\mu) = \inf_{x\in H}\mu(\intr{H}),	\]
considered in, e.g., \cite[Lemma~1]{Mizera_Volauf2002}. Note also that $(\Ucamu)\compl=\bigcup\{H\in \half \colon \mu(\intr{H})\leq \alpha\}$ already appeared in Theorem~\ref{theorem:ray basis additional}, meaning that $\Ucamu=\intr{\Uamu}$ must be an open set. The following theorem comprehensively covers the inter-relations between all these affine constructions, and generalizes several results that can be found scattered in the relevant literature on multivariate trimming concepts \cite{Small1987, Masse_Theodorescu1994, Mizera_Volauf2002, Brunel2019, Nagy_etal2019}.

\begin{theorem} \label{theorem:intersection}
For $\mu\in\Meas$ and $\alpha\in \R$ 
	\begin{equation}	\label{inclusions}
	\begin{aligned}
	\Ucamu = \intr{\Uamu} \subseteq \cl{\Uamu} \subseteq \UaFBmu \subseteq  \Damu.
	\end{aligned}
	\end{equation}
Additionally,
	\begin{enumerate}[label=(\roman*), ref=(\roman*)]
	\item \label{inclusions null}  $\Ucamu = \left\{x\in\R^d \colon \D^\circ(x;\mu) > \alpha\right\}$. 
	\item \label{inclusions first} if $\intr{\Uamu}\neq \emptyset$, then $\cl{\Uamu} = \UaFBmu$. 
	\item \label{inclusions second} if $\mu$ has contiguous support at $\Damu$, then $\intr{\Damu}=\intr{\Uamu}$.
	\item \label{inclusions floating body} if for each $H^\prime \in \half(\Damu)$ we have that $H^\prime \subset H \in \half$ implies $\mu(H^\prime) < \mu(H)$, then $\UaFBmu = \Damu$.
	\item \label{strict monotonicity} if $\intr{\Uamu}\neq \emptyset$ and $\mu$ has contiguous support at $\Damu$, then $\cl{\Uamu} = \Damu$.
	\end{enumerate}
\end{theorem}

An application of Theorem~\ref{theorem:intersection} to the estimation of the depth regions $\Damu$ from datasets is given in Section~\ref{section:Dyckerhoff}. The condition of non-empty interior of $\Uamu$ that figures in parts~\ref{inclusions first} and~\ref{strict monotonicity} of the previous theorem is not restrictive, as shown in the next lemma. 

\begin{lemma}	\label{lemma:strict monotonicity condition}
If for $\mu\in\Meas$ and $\alpha \in \R$ there exists a point $x\in\Uamu$ such that $\mu$ is smooth at $x$, then $\intr{\Uamu}\neq \emptyset$.
\end{lemma}

Combining Lemma~\ref{lemma:strict monotonicity condition} and Theorems~\ref{theorem:ray basis additional} and~\ref{theorem:intersection} we get that for $\mu\in\Meas$ with a density and $\alpha \in (0,\alpha^*(\mu))$ the set $\Uamu$ is open and $\Uamu \subset \cl{\Uamu} = \UaFBmu$. If the support of $\mu$ is, in addition, contiguous, we can also write $\UaFBmu = \Damu$. In the simplest situation of $\mu\in\Meas$ with a density that is positive in the convex support of $\mu$, we can therefore write $\cl{\Uamu} = \UaFBmu = \Damu$, and the floating bodies completely coincide with the central regions of the depth. The last situation is common in the literature on floating bodies in geometry, where all the above definitions are used interchangeably. In the general setup of the depth and (probability) measures, it is however necessary to differentiate between them.

Our statement of Theorem~\ref{theorem:intersection} is strict --- none of the inclusions can be reversed, in general. Each strict upper level set $\Uamu$ of a smooth measure $\mu$ is open (Theorem~\ref{theorem:ray basis additional}), and thus strictly smaller than $\cl{\Uamu}$. It is easy to construct a measure without contiguous support that violates $\UaFBmu = \Damu$, see e.g. \cite[Figure~7]{Nagy_etal2019}. Even for measures with contiguous support, equality $\UaFBmu = \Damu$ from part~\ref{inclusions floating body} of Theorem~\ref{theorem:intersection} can fail; consider $\mu$ being the Dirac measure at the origin $o \in \R^d$, and $\alpha = 1$. In that case, $\UaFBmu = \emptyset$, while $\Damu = \{ o \}$. In the following example we show that in the case when $\intr{\Uamu}=\emptyset$, the halfspace depth may fail to satisfy $\cl{\Uamu} = \Damu$ (the so-called strict monotonicity property), even under the assumption of contiguous support. The example demonstrates that also the remaining set inclusions in Theorem~\ref{theorem:intersection} cannot be reversed for general measures without local assumptions.

\begin{example}	\label{example:strict monotonicity}
Consider $\mu \in \Meas[\R^2]$ given by a mixture of the uniform distribution on the unit disc $C = \left\{ x \in \R^2 \colon \left\Vert x \right\Vert \leq 1 \right\}$ with mass $1/4$, and two atoms located at $m = \left(0,1\right)\tr$ and $z = \left(0,1/2\right)\tr$ with masses $1/2$ and $1/4$, respectively. Point $m$ is certainly the unique halfspace median of $\mu$, with $\D\left(m;\mu\right) = \alpha^*(\mu) = 1/2$. It is also easy to see that the halfspace depth of $\mu$ at the point $o = \left(0,0\right)\tr$ equals $1/8$. Because of the two large atoms of $\mu$, all points in the open line segment $L(o,z)$ have also depth $1/8$. Similarly, all points in $L(z,m)$ have the same depth as $z$, that is $\D\left(z;\mu\right) = 1/8 + 1/4 = 3/8$. Because for any $y \notin L[o,m]$ certainly $\D\left(y;\mu\right) < 1/8$, we have for $\alpha = 1/8$ that $\Damu = L[o,m]$, while $\cl{\Uamu} = L[z,m]$, meaning that the strict monotonicity property of the halfspace depth is violated, see also the left hand panel of Figure~\ref{fig:circle}. In addition, it is easy to verify that $\UaFBmu$ is equal to $\Damu \supset \cl{\Uamu}$, while the set $\Ucamu$ from \eqref{inclusions} is empty. We conclude that, in general, one cannot write $\Uamu$ as a simple intersection of halfspaces, as possible for $\Damu$ in \eqref{intersection of closed halfspaces}.
\end{example}

\begin{figure}[htpb]
\includegraphics[width=\twofig\textwidth]{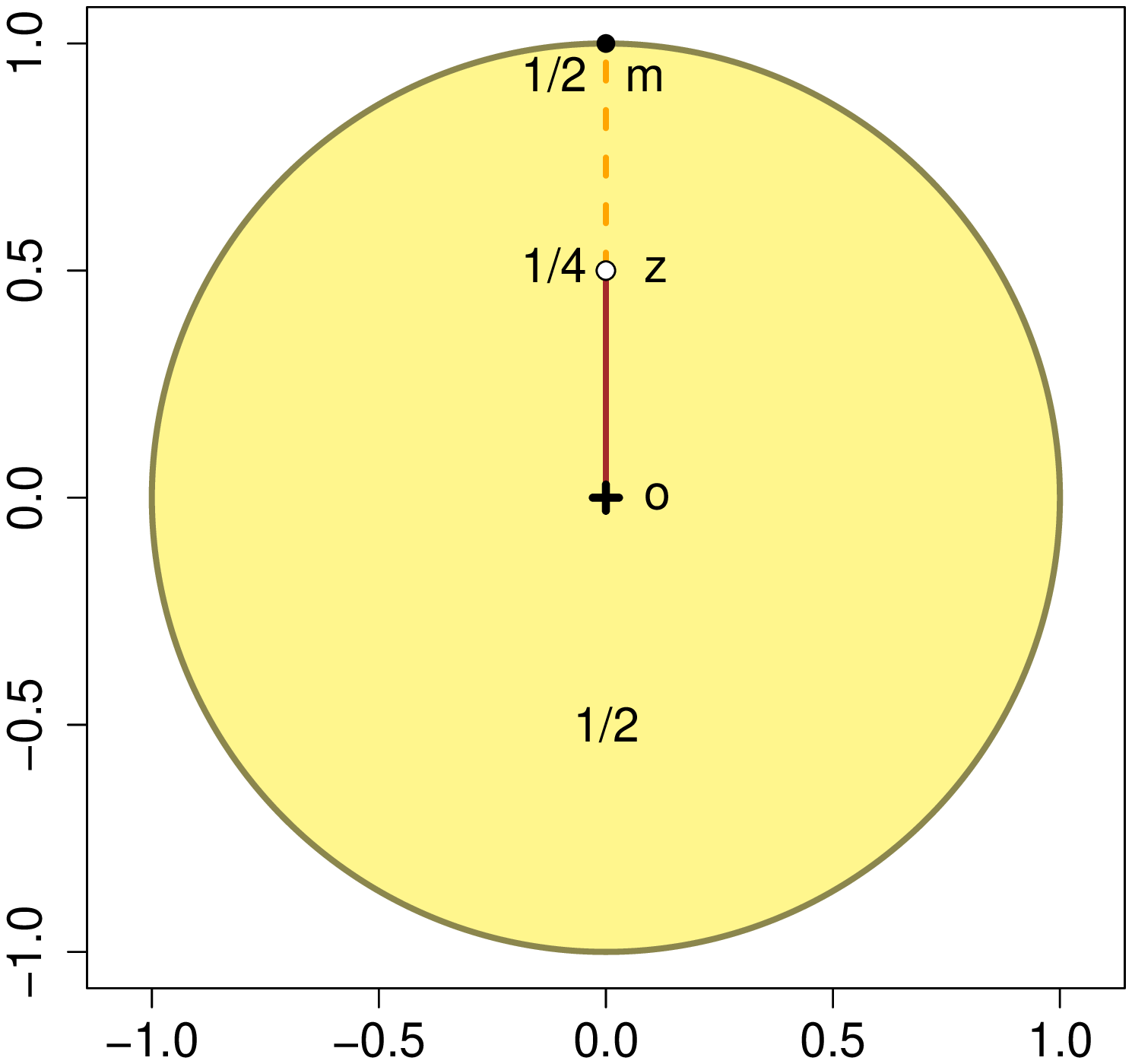} \quad
\includegraphics[width=\twofig\textwidth]{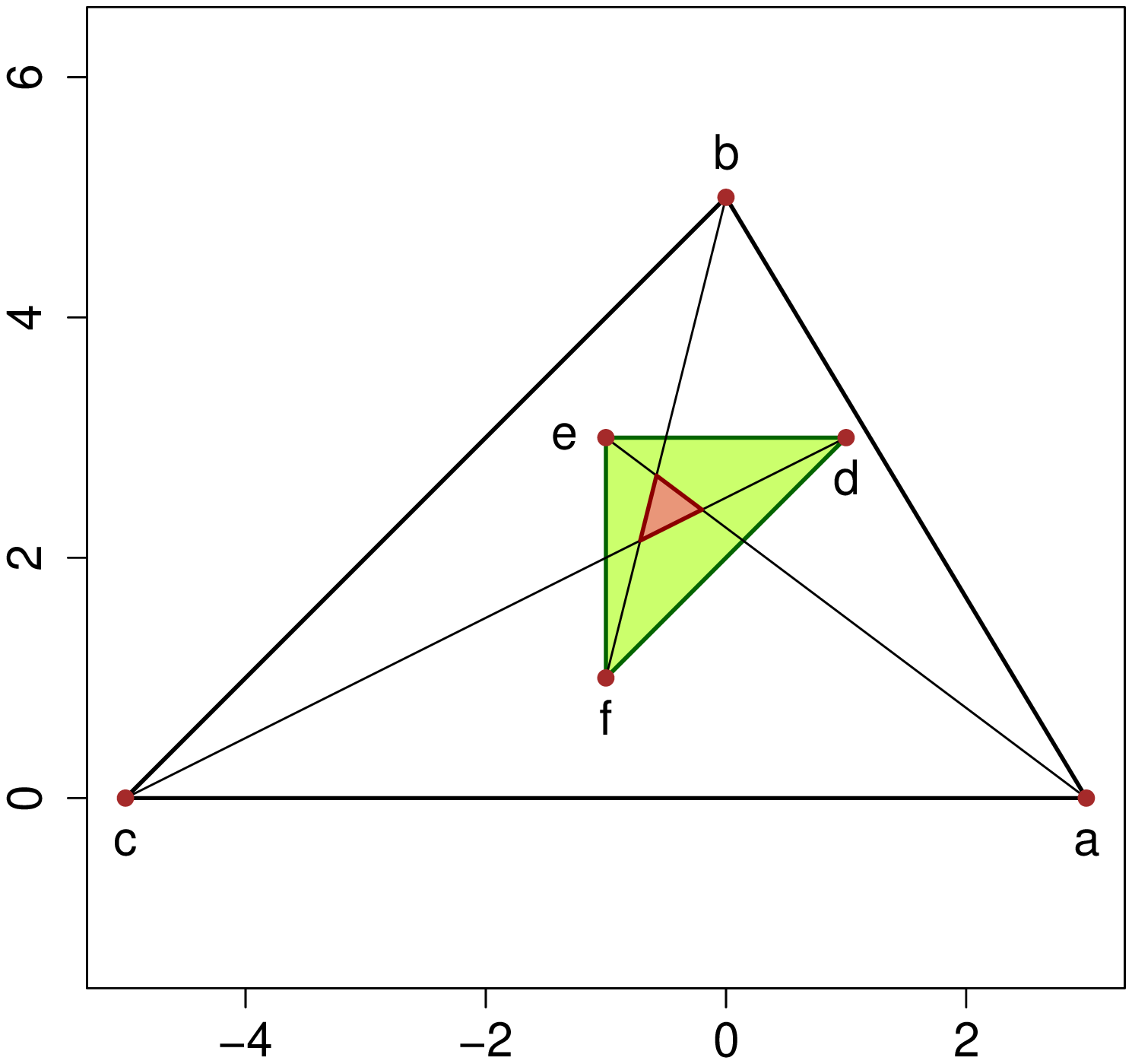}
\caption{Left panel: Condition $\intr{\Uamu}\neq \emptyset$ is needed for the strict monotonicity of the halfspace depth: Example~\ref{example:strict monotonicity} of a distribution $\mu$ where $\Damu \ne \cl{\Uamu}$. The halfspace depth of $\mu$ takes only values in the set $[0,1/8] \cup \{3/8, 1/2\}$. The upper level set $\Damu$ with $\alpha = 1/8$ is the line segment between the origin $o$ and the black atom $m$. Points on the thick dashed line all attain depth equal to $3/8$. 
Right panel: For the atomic measure $\mu$ supported in points $a$--$f$ from Example~\ref{example:double triangle}, the set $\coverM$ (innermost coloured triangle) is a proper subset of $\median$ (outer coloured triangle).
}
\label{fig:circle}
\end{figure}

\section{Applications: Refined medians, computation, and consistency}	\label{section:applications}

\subsection{Covering halfspace median}	\label{section:covering median}

Our first application of the ray basis theorems from Section~\ref{section:ray basis} is motivated by another refinement of the main equality \eqref{D equality} applied to $\alpha = \alpha^*(\mu)$: For any measure $\mu \in \Meas$ there exists a point $x\in \median$ that allows us to cover the whole space with halfspaces $H\in\half (x)$ whose interior has $\mu$-mass at most 
	\[	\gamma^*(\mu) = \inf\{\alpha>0\colon \intr{D_{\alpha}(\mu)}=\emptyset\} = \sup\{\alpha>0\colon \intr{D_{\alpha}(\mu)}\ne\emptyset\}.	\]
Observe that $\gamma^*(\mu) \leq \alpha^*(\mu)$. Therefore, the following claim is stronger than the inverse ray basis theorem not only \begin{enumerate*}[label=(\roman*)] \item because it asserts the existence of such a special point in $\median$, but also in the sense that \item the masses of the halfspaces that cover $\R^d$ are ensured to be at most $\gamma^*(\mu)$, and not just $\alpha^*(\mu)$. \end{enumerate*}

\begin{theorem} \label{theorem:covering point in median}
For any $\mu \in \Meas$ there exists a point $x\in \median$, such that
	\begin{equation} 	\label{eq:covering property}
	\R^d = \bigcup\left\{	H\in\half(x) \colon \mu(\intr{H}) \leq \gamma^*(\mu)	\right\}.
	\end{equation}
\end{theorem}
 
Suppose that $\median$ is not a single point set. It is certainly of interest to single out the subset $\coverM$ of $\median$ of those points that satisfy the additional covering property \eqref{eq:covering property} characteristic to centrally located points. We call such medians the \emph{covering halfspace medians} (or simply \emph{covering medians}) of the measure $\mu$. The concept of covering medians is interesting especially in the situation when the median set is full-dimensional, as frequently happens with data generated from a smooth distribution. In that situation, the relatively large median set typically reduces to a smaller subset of the most centrally located points being the covering medians.

\begin{example}	\label{example:double triangle}
For an empirical measure $\mu\in\Meas[\R^2]$ with atoms at points $a=\left(3,0\right)\tr$, $b=\left(0,5\right)\tr$, $c=\left(-5,0\right)\tr$, $d=\left(1,3\right)\tr$, $e=\left(-1,3\right)\tr$ and $f=\left(-1,1\right)\tr$, the median set $\median$ is equal to the triangle determined by points $d$, $e$ and $f$, while $\coverM$ is a smaller triangle determined by the lines $l(a,e)$, $l(b,f)$ and $l(c,d)$, see the right hand panel of Figure~\ref{fig:circle}.
\end{example}

\begin{figure}[htpb]
\includegraphics[width=\onefig\textwidth]{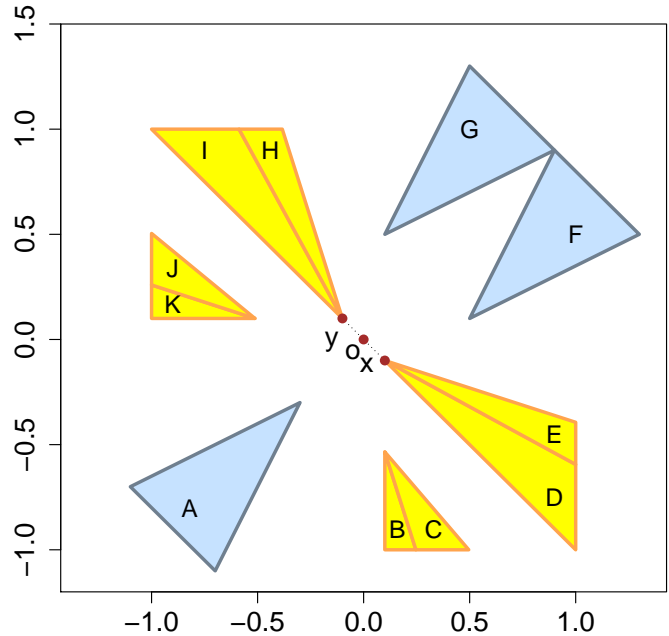}
\caption{Example~\ref{example:not convex}: The set of covering medians $\coverM$ may fail to be convex. For the measure $\mu$, whose support is displayed in this plot in colour, the points $x$ and $y$ belong to $\coverM$. But, their midpoint $o$ is not contained in $\coverM$.}
\label{figure:square}
\end{figure}

The covering medians present a genuine refinement of the halfspace medians. They can be shown to satisfy an array of properties expected from well-behaved location estimators, such as \begin{enumerate*}[label=(\roman*)] \item existence as proved in Theorem~\ref{theorem:covering point in median}; \item affine equivariance; or \item plausible continuity properties when considered as a set-valued function of the measure $\mu$. \end{enumerate*} The covering medians are also intimately connected with the robustness properties of the halfspace median and the depth. All these results will be presented elsewhere in an appropriate context. Here we mention only several basic observations closely linked to Theorem~\ref{theorem:covering point in median}. The first one demonstrates that unlike the standard halfspace median set, the set of the covering medians does not have to be convex.

\begin{example}	\label{example:not convex}
Let $\mu\in \Meas[\R^2]$ be the mixture of uniform distributions on the triangles \emph{A}--\emph{K} displayed in Figure~\ref{figure:square}, such that $\mu(A)=\mu(F)=\mu(G)=2$ and $\mu(B)=\mu(C)=\mu(D)=\mu(E)=\mu(H)=\mu(I)=\mu(J)=\mu(K)=1$. It can be shown that the set $\coverM$ contains points $x$ and $y$ displayed in Figure~\ref{figure:square}, but it does not contain the origin $o=(x+y)/2$. Therefore, the set of the covering halfspace medians of $\mu$ is not convex. For a detailed proof see the Appendix, Section~\ref{section:not convex}.
\end{example}

The proof of Theorem~\ref{theorem:covering point in median} allows us to devise a simple algorithm for finding the covering medians of $\mu$ listed as Algorithm~\ref{algorithm:covering median}. Although this program is applicable to any measure including empirical measures of random samples, its main purpose is not the computation of the sample covering medians of large datasets. Rather, it is intended to guide a quick manual procedure for restricting the location of possible covering medians in visual examples such as those presented throughout this paper.
\begin{algorithm}[htpb]
\SetAlgoLined
\SetKwInOut{Input}{input}\SetKwInOut{Output}{output}

\Input{the full-dimensional median region $\median$ of a measure $\mu$}
\Input{a small positive constant $\varepsilon$ determining desired precision}
\Output{a covering median of $\mu$}
	\BlankLine
	$S_0 \leftarrow \median$ \;
	$x_0 \leftarrow $ barycentre of $S_0$ \;
	$k \leftarrow 0$ \;
	\While{$x_k\notin\coverM$ and volume of $S_k$ exceeds $\varepsilon$}{
		$H_k \leftarrow $ generalized minimizing halfspace of $\mu$ at $x_k$ \;
		\tcp{such a halfspace exists by Lemma~\ref{lemma:minimizing halfspace}}
		$S_{k+1} \leftarrow \cl{S_k \setminus H_k}$ \;
		$k \leftarrow k+1$ \;
		$x_k \leftarrow $ barycentre of $S_k$ \;
	}
	\Return{$x_k$} \;
 \caption{Search for a covering halfspace median of a measure $\mu$.}
 \label{algorithm:covering median}
\end{algorithm}

Once a covering median of $\mu$ is found, it is of interest do determine whether it is unique. Supposing that a covering median of $\mu$ and its collection of generalized minimizing halfspaces are available as the output of Algorithm~\ref{algorithm:covering median}, the following theorem gives a sufficient condition for the uniqueness of this covering median.

\begin{theorem} \label{unique covering median}
Let $x\in\R^d$ be a covering median of $\mu \in \Meas$. Denote by $\half_{\min}$ the collection of halfspaces on the right hand side of the following display
	\begin{equation*} 
	\R^d=\bigcup\left\{H\in\half(x)\colon\mu(\intr{H})\leq \gamma^*(\mu)\right\}.
	\end{equation*}
If for each $H\in \half_{\min}$ and $H^\prime \supset H$ it follows that $\mu(H^\prime)>\mu(H)$, and if there is no subset $\half^\prime \subseteq \half_{\min}$ such that $\bigcup \half^\prime = \R^d$ and $\bigcap \half^\prime \neq \{x\}$, then $\coverM=\{x\}$.
\end{theorem}

Note that the point $y_c$ from Example~\ref{example:triangle gap} given in Section~\ref{section:level sets} satisfies the conditions of Theorem~\ref{unique covering median} and $\coverM=\{y_c\}$, while the median of that measure $\mu$ is full-dimensional. On the other hand, Theorem~\ref{unique covering median} does not apply to any of the two covering medians $x, y$ found in Example~\ref{example:not convex}, as the condition regarding contiguous support at minimizing halfspaces is not satisfied.

%
%

\subsection{Trimmed regions for atomic measures}	\label{section:computation}	

Our second application of the generalized ray basis theorem is a necessary condition on the depth regions $\Damu$ of atomic measures. 

\begin{corollary} \label{Ray basis corollary}
For $\mu$ an atomic measure with finitely many atoms, $\alpha \in \R$, $x \notin \intr{\Damu}$ and $F\in \mathfrak{F}(x,\Damu)$ there exists $H(x,F)\in\half(\Damu,F)$ such that $x\in H(x,F)$ and $\mu(\intr{H(x,F)})< \alpha$. Moreover, each face $F$ of the convex polytope $\Damu$ of dimension $\dim(F) < d$ is contained in the convex hull of at least $\min\{\dim(\Damu)+1,d\}$ atoms located in a hyperplane in $\R^d$. 
\end{corollary}

A special case of this result in the situation when $\Damu$ is full-dimensional was used in the derivation of a fast algorithm for the computation of the depth regions of datasets in \cite{Liu_etal2019}. The present general version may find applications in the computation of the halfspace depth of datasets in the situation when $\dim\left(\Damu\right) < d$. For example, suppose that for given $\alpha \in \R$ and $\mu$ corresponding to a dataset the algorithm from \cite{Liu_etal2019} fails to find an interior point of the region $\Damu$. The reason may be twofold: either \begin{enumerate*}[label=(\roman*)] \item $\Damu$ is an empty set, or \item it is less than full-dimensional.\end{enumerate*} Corollary~\ref{Ray basis corollary} asserts that in the latter case, the region $\Damu$ must be contained in a hyperplane spanned by $d$ data points. Thus, the search for $\Damu$ may continue in the intersection of data-determined hyperplanes, and the last known non-empty region $\D_\beta(\mu)$ for $\beta < \alpha$. 

%
%

\subsection{Estimation of central regions}	\label{section:Dyckerhoff}

Part~\ref{strict monotonicity} of Theorem~\ref{theorem:intersection} yields an important equality. According to \citet[Definition~3.1]{Dyckerhoff2017} the halfspace depth is \emph{strictly monotone} at $\mu \in \Meas$ if for all $\alpha \in (0,\alpha^*(\mu))$ we can write $\Damu = \cl{\Uamu}$. Strict monotonicity is a crucial assumption that ensures the sample version consistency of the halfspace depth-trimmed regions \cite[Example~4.2]{Dyckerhoff2017}. As a consequence of Lemma~\ref{lemma:strict monotonicity condition} we obtain a condition ensuring the strict monotonicity of the halfspace depth --- for a measure $\mu\in\Meas$ to have a strictly monotone depth it is enough to assume contiguous support and smoothness at a single point of the median set. In particular, under these very mild conditions, it is possible to guarantee the almost sure uniform convergence of the sample depth trimmed regions to their population counterparts. To precise this, we need to consider a topology on the space of compact subsets of $\R^d$. A natural choice is that given by the Hausdorff distance \cite[Section~1.8]{Schneider2014}. For $K, L \subset \R^d$ compact the Hausdorff distance of $K$ and $L$ is given by
		\[	\Haus(K,L) = \max\left\{ \sup_{x \in K} \inf_{y \in L} \left\vert x - y \right\vert, \sup_{x \in L} \inf_{y \in K} \left\vert x - y \right\vert \right\}.	\]

\begin{corollary}	\label{corollary:Dyckerhoff}
Let $\mu\in\Meas$ be a probability measure with contiguous support that is smooth at some $x \in \median$. Let $X_1, \dots, X_n$ be a random sample from $\mu$ defined on the probability space $\left(\Omega, \mathcal A, \PP\right)$, and denote by $\widetilde{\mu_n} \equiv \widetilde{\mu_n}(\omega) \in \Meas$ the empirical measure of $X_1, \dots, X_n$. Then for any closed interval $A \subset \left( 0, \alpha^*(\mu) \right)$
	\[	\PP\left( \left\{ \omega \in \Omega \colon \lim_{n\to\infty} \sup_{\alpha \in A} \Haus\left( \D_{\alpha}(\widetilde{\mu_n}), \Damu \right) = 0 \right\}	\right) = 1.	\]
Further, let $\mu\in\Meas$ be any measure that is smooth with contiguous support. Then for any closed interval $A \subset \left( 0, \alpha^*(\mu) \right)$ and any $\mu_n$ converging weakly to $\mu$ we have
	\[	\lim_{n\to\infty} \sup_{\alpha \in A} \Haus\left( \D_{\alpha}(\mu_n), \Damu \right) = 0. \]
\end{corollary}

This result follows directly from \cite[Theorem~4.5]{Dyckerhoff2017} and our previous discussion. The \emph{General assumption} in \cite[p.~9]{Dyckerhoff2017} is satisfied thanks to our Lemma~\ref{lemma:strict monotonicity condition}; the assumption of compact convergence of the depth from \cite{Dyckerhoff2017} follows, e.g., from \cite[Section~3.2.7]{Nagy_etal2019}. Corollary~\ref{corollary:Dyckerhoff} should be compared with earlier contributions regarding the consistency of the trimmed regions and derived quantities \cite{Nolan1992, He_Wang1997, Zuo_Serfling2000b, Kim2000, Masse2002, Wang_Serfling2006, Wang2019}. In those references analogous consistency results are proved under more restrictive conditions.

%
%
%
%

\appendix

\section{Proofs of the theoretical results}

We begin with a lemma collecting several properties of the convergence of halfspaces.

\begin{lemma}	\label{lemma:convergent sequence}
Consider $\mu\in \Meas$, a sequence of closed halfspaces $\{H_n\}_{n=1}^\infty \subset \half$, and $\alpha \in \R$. The following claims hold true:
	\begin{enumerate}[label=(\roman*), ref=(\roman*)]
	\item \label{convergence of halfspaces} If there is a sequence $\{x_n\}_{n=1}^\infty$ such that $x_n\in \bd{H_n}$ and $x_n\rightarrow x$, then there exists a subsequence $\{H_{n_k}\}_{k=1}^\infty$ converging to a closed halfspace $H\in\half(x)$.
	\item \label{smaller or equal alpha} If $H_n\rightarrow H$ and $\mu(\intr{H_n})\leq \alpha$ for each $n$, then $\mu(\intr{H})\leq \alpha$.
	\item \label{equal to alpha} If $H_n\rightarrow H$, $\mu(\intr{H_n})\leq \alpha$ for each $n$, $H\cap \Damu\neq \emptyset$ and $\mu(\bd{H})=0$, then $\mu(H)=\alpha$.
	\item \label{converge to touching} If $H_n\rightarrow H$ and $x \notin \intr{H_n}$ for each $n$, then $x \notin \intr{H}$.
	\end{enumerate}
\end{lemma}

\begin{proof}
For part \ref{convergence of halfspaces}, denote $H_n = H_{x_n, v_n}$. The set of unit vectors $\Sph$ is bounded, meaning that $\left\{ v_n \right\}_{n=1}^\infty \subset \Sph$ contains a convergent subsequence $\{v_{n_k}\}_{k=1}^\infty$, whose limit point we denote by $v \in \Sph$. We obtain that $H_{n_k}\rightarrow H_{x,v} = H \in \half(x)$ as $k \to \infty$ as desired. Part \ref{smaller or equal alpha} is a consequence of the Fatou lemma \cite[Lemma~4.3.3]{Dudley2002} and the fact that $\intr{H} \subseteq {\lim\inf}_{k\to\infty} \intr{H_{n_k}}$. Under the assumptions of part~\ref{equal to alpha} there exists $x \in \Damu \cap H$. For this point we can directly write by part \ref{smaller or equal alpha} of this lemma $\alpha \leq \D\left(x;\mu\right) \leq \mu(H) = \mu(\intr{H}) \leq \alpha$. To prove \ref{converge to touching}, assume that $x\in\intr{H}$. Then $x\in\intr{H_n}$ starting from some index, a contradiction.
\end{proof}

\begin{lemma}	\label{lemma:boundary of halfspace}
If $H\in\half$ and $S \subset \R^d$ is a convex set such that $\intr{H} \cap S = \emptyset$ and $\relint{S} \cap H \neq \emptyset$, then $\aff{S}\subseteq \bd{H}$.
\end{lemma}

\begin{proof} 
Consider $x \in \relint{S} \cap H$. Since $x \in \relint{S}$, there exists a ball $B_x$ centred at $x$ in space $\aff{S}$ so that $B_x \subset S$. Since $\intr{H} \cap S = \emptyset$ we get $x \in \bd{H}$ and $B_x \subset \bd{H}$, meaning that $\aff{S} = \aff{B_x} \subseteq \bd{H}$.
\end{proof}

\begin{lemma}	\label{lemma:faces} 
For $A \subset \R^d$ non-empty convex and $x\notin \intr{A}$ the collection $\mathfrak{F}(x,A)$ contains a non-empty set.
\end{lemma}

\begin{proof}
In the case that $x\in \bd{A}$, there is a face $F \ne \emptyset$ of $\bd{A}$ containing $x$ in its relative interior \cite[Theorem~2.1.2]{Schneider2014}, so $\mathfrak{F}(x,A)$ contains a non-empty set. Otherwise, let $x\notin \cl{A}$. The Hahn-Banach theorem \cite[Theorem~1.3.4]{Schneider2014} guarantees that there exists a touching halfspace $H \in \half(A)$ of $\cl{A}$ such that $x\in \intr{H}$. Denote $F=\cl{A} \cap \bd{H}$. Then, $F$ is a non-empty face of $\cl{A}$, and certainly also $\cl{A} \setminus F \subseteq H\compl$, $F \subseteq \bd{H}$, and $x \in \intr{H}$, meaning that $F\in \mathfrak{F}(x,A)$. 
\end{proof}

\subsection{Proof of Lemma~\ref{lemma:minimizing halfspace}}

Let $\alpha = \D\left(x;\mu\right)$. The definition of the halfspace depth ensures that for each $n = 1, 2, \dots$ there exists $H_n \in \half(x)$ such that $\mu(H_n) < \alpha+1/n$. Applying part~\ref{convergence of halfspaces} of Lemma~\ref{lemma:convergent sequence} we obtain a subsequence $\{H_{n_k}\}_{k=1}^\infty$ converging to $H\in\half(x)$. From the Fatou lemma \cite[Lemma~4.3.3]{Dudley2002} we get $\mu(\intr{H}) \leq {\lim\inf}_{k\to\infty} \mu(\intr{H_{n_k}}) \leq \alpha$.

Under the additional assumption of smoothness of $\mu$ at $x$ we know that $\mu\left(\bd{H}\right) = \mu\left(\bd{H_n}\right) = 0$ for each $n=1,2,\dots$. In that case, the Fatou lemma guarantees $\mu(H) = \lim_{k \to \infty} \mu(H_{n,k}) \leq \alpha$. Since also $\alpha = \D\left(x;\mu\right) \leq \mu(H) \leq \alpha$, we obtain the desired result.

\subsection{Proof of Lemma~\ref{lemma:basic condition}}

For any $S \subset \R^d$ convex and $H \in \half$, the condition $\intr{H} \cap S = \emptyset$ implies $\intr{H} \cap \intr{S} = \emptyset$. Furthermore $\intr{S} \subseteq S$. Thus, it is enough to show the claim \eqref{basic condition} with the smaller set $\intr{S}$ on the left hand side, and the condition $\intr{H} \cap S = \emptyset$ in the expression on the right hand side. 

We start with $\Damu$. First note that by the very definition of $\Damu$
	\begin{equation}	\label{empty intersection}
	\mu(\intr{H})<\alpha \mbox{ implies }\intr{H} \cap \Damu=\emptyset.	
	\end{equation}
Consider any $x\notin \intr{\Damu}$. Suppose first that $x\notin \Damu$, meaning that $\D(x;\mu)<\alpha$. The auxiliary Lemma~\ref{lemma:minimizing halfspace} implies that there exists $H\in \half(x)$ such that $\mu(\intr{H})<\alpha$, and \eqref{empty intersection} ensures that $x$ is covered by the union on the right hand side of \eqref{basic condition}. If $x \in \bd{\Damu}$, we can approach $x$ by a sequence $\{x_n\}_{n=1}^\infty\subset \left(\Damu\right)\compl$ such that $x_n \to x$. We apply Lemma~\ref{lemma:minimizing halfspace} to each $x_n$ to obtain halfspaces $H_n\in\half(x_n)$ such that $\mu(\intr{H_n})<\alpha$, and consequently $\intr{H_n}\cap \Damu=\emptyset$ for each $n$ by \eqref{empty intersection}. From Lemma~\ref{lemma:convergent sequence}, parts~\ref{convergence of halfspaces}, \ref{smaller or equal alpha} and~\ref{converge to touching} we conclude that there exists $H\in \half(x)$ such that $\mu(\intr{H})\leq\alpha$ and $\intr{H} \cap \Damu=\emptyset$, as desired.

The proof for $\Uamu$ is analogous. First, observe that $\mu(\intr{H})\leq \alpha$ implies $\intr{H}\cap \Uamu=\emptyset$ by the definition of $\Uamu$. Consider any $x\notin \intr{\Uamu}$. If $x\notin \Uamu$, then $\D(x;\mu)\leq \alpha$, so there is $H\in \half(x)$ such that $\mu(\intr{H})\leq \alpha$ by Lemma~\ref{lemma:minimizing halfspace}. Otherwise, $x \in \bd{\Uamu}$ may be approached by a sequence $\{x_n\}_{n=1}^\infty\subset \left(\Uamu\right)\compl$ such that $x_n \to x$, and Lemma~\ref{lemma:convergent sequence} again gives the desired result.

\subsection{Proof of Lemma~\ref{lemma:basic lemma}}

Since each $H \in \half(S)$ has empty intersection with $\intr{S}$, it is only left to prove $(\intr{S})\compl\subseteq \bigcup\left\{ H \in \half(S) \colon \mu(\intr{H}) \leq \alpha \right\}$. We take any $x \notin \intr{S}$ and consider three cases.

\textbf{Case (i): $\intr{S} = \emptyset$.} By definition, $\half\left(\emptyset\right) = \half$. Therefore, the statement reduces directly to condition \eqref{basic condition}.

\textbf{Case (ii): $x \notin \cl{S}$ and $\intr{S} \ne \emptyset$.} We consider only the situation when $F$ is non-empty, which is possible by Lemma~\ref{lemma:faces}, as the other case follows trivially from condition \eqref{basic condition}. Denote $C=\relint{\conv{F\cup \{x\}}}$. Take $y \in \relint{F}$, and the sequence $y_n = \left(1 - 1/n\right)y + x/n \in C$, for $n = 1, 2, \dots$, that converges to $y$. Because $x \notin S$, the choice of $F$ ensures that $C\cap \cl{S}= \emptyset$. Thus, for each $n$ we have $y_n\notin \cl{S}$ and condition~\eqref{basic condition} implies the existence of $H_n \in \half$ such that $y_n \in H_n$, $\intr{H_n} \cap S = \emptyset$ and $\mu\left(\intr{H_n}\right) \leq \alpha$. Applying Lemma~\ref{lemma:convergent sequence}, parts~\ref{convergence of halfspaces}, \ref{smaller or equal alpha} and~\ref{converge to touching}, we conclude that there is a subsequence $\{H_{n_k}\}_{k=1}^\infty$ converging to $H \in \half(y)$, such that $\mu(\intr{H})\leq \alpha$ and $\intr{H}\cap S=\emptyset$. Then $y\in \relint{F}\cap H$. From Lemma~\ref{lemma:boundary of halfspace} it follows that $F\subset \bd{H}$, so $H\in\half (S,F)$.

\textbf{Case (iii): $x \in \bd{S}$ and $\intr{S} \ne \emptyset$.} The face $F$ was chosen so that $x \in \relint{F}$. It is enough to follow the lines of Case (ii) of this proof with $y$ replaced by $x$, and $x$ replaced by $z = x + u \notin S$ for $u \in \Sph$ (any) outer normal to $S$ at $x$. The latter choice assures that $F \in \mathfrak F(y,S)$, which allows us to proceed exactly as in Case (ii).

\subsection{Proof of Theorem~\ref{theorem:ray basis depth regions}}
The general statement follows from Lemmas~\ref{lemma:basic condition} and~\ref{lemma:basic lemma}. Thus, it is enough to prove only the results under additional assumptions. We take $x \in \R^d$ and consider the construction of the halfspace $H = H(x,F) \in \half(\Damu, F)$ from the proof of Lemma~\ref{lemma:basic lemma}. Only the situation when both $\Damu$ and $F$ are non-empty is considered; the other case is straightforward and trivial.

\textbf{Part~\ref{smooth case}:} Under the assumption of smoothness at $\Damu$ we have $\mu(H(x,F))=\mu(\intr{H(x,F)})$. Since $\emptyset \ne F \subseteq \Damu$ we also have that $H(x,F) \cap \Damu \neq \emptyset$, which implies $\mu(H(x,F))=\alpha$.

\textbf{Part~\ref{contiguous support case}:} It is enough to show that for any $H\in \half$ the condition $\mu(\intr{H})\leq \alpha$ implies $\intr{\Damu} \cap H = \emptyset$. Suppose $x\in H\cap \intr{\Damu}$ and denote by $v$ the unit inner normal of $H$. Then $H_{x,v}\subseteq H$ and $\bd{H_{x,v}}\cap \intr{\Damu}\neq \emptyset$. Since $\Damu$ is closed and convex, there exists $y\in \intr{H_{x,v}}\cap \Damu$, such that $H_{y,v} \in \half(\Damu)$, and $H_{y,v}\subset H_{x,v}$. Because $\mu$ has contiguous support at $\Damu$ it follows that $\mu(H_{y,v})<\mu(H_{x,v})\leq \mu(H)\leq \alpha$, which is in contradiction with $y\in \Damu$.

\textbf{Part~\ref{smooth and contiguous support case}:} The result follows from parts~\ref{smooth case} and~\ref{contiguous support case}.

\subsection{Proof of Theorem~\ref{theorem:ray basis additional}}

The inclusion ``$\subseteq$" in \eqref{U equality} is a direct consequence of Lemmas~\ref{lemma:basic condition} and~\ref{lemma:basic lemma}. For the other inclusion, suppose for contradiction that $x \in \intr{\Uamu}$ is contained in the right hand side of \eqref{U equality}. Then we can find $H \in \half$ with $x \in H$ and $\mu(\intr{H}) \leq \alpha$, and since both $\intr{H}$ and $\intr{\Uamu}$ are open, there must exist $y \in \intr{\Uamu} \cap \intr{H}$. Denote by $v \in \Sph$ the inner unit normal of $H$. We obtain that $H_{y,v} \subset \intr{H}$, meaning that $\alpha < \D\left(y;\mu\right) \leq \mu\left(H_{y,v}\right) \leq \mu(\intr{H}) \leq \alpha$, a contradiction. We have verified \eqref{U equality}.

If $\mu$ is smooth at $\Uamu$, then for $H(x,F) \in \half(\Uamu,F)$ from the proof of Lemma~\ref{lemma:basic lemma} we have $\mu(H(x,F))=\mu(\intr{H(x,F)})\leq \alpha$ and consequently $H(x,F)\cap \Uamu=\emptyset$, meaning that $F\cap \Uamu=\emptyset$. This holds true for any $x \notin \intr{\Uamu}$ and any face $F$ from $\mathfrak F(x,\Uamu)$. In particular, it must be true for any face $F$ of $\Uamu$ that is not full-dimensional, by considering $x \in \relint{F} \subset \bd{\Uamu}$. That implies $\Uamu \cap \bd{\Uamu} = \emptyset$, meaning that $\Uamu=\intr{\Uamu}$ and $\Uamu$ is open.

\subsection{Proof of Corollary~\ref{corollary:dimension of median}}	\label{section:proof of dimension}

\textbf{Part~\ref{dimension i}:} Suppose for contradiction that $x \in \intr{\median}$. Lemma~\ref{lemma:minimizing halfspace} gives $H \in \half(x)$ with $\mu(\intr{H}) \leq \alpha^*(\mu)$. We can shift $H$ in the direction of its inner normal to obtain $H^\prime \in \half\left(\median\right)$ such that $H^\prime \subset \intr{H}$. Since $\median$ is closed, there exists $y \in \median \cap H^\prime$. The assumption of contiguous support at $\median$ then gives $\D\left(y;\mu\right) \leq \mu(H^\prime) < \mu(\intr{H}) = \alpha^*(\mu)$, which is impossible as $y \in \median$.

\textbf{Part~\ref{dimension ii}:} If $\dim\left(\median\right) = d-1$, for any $x$ not lying in the hyperplane $\aff{\median}$ we have $\median \in \mathfrak{F}(x,\median)$. Theorem~\ref{theorem:ray basis depth regions} gives that there exists a halfspace $H \in \half$ that contains $\median$ in its boundary hyperplane, and $x$ in its interior, with $\mu\left(H\right) = \mu\left(\intr{H}\right) \leq \alpha^*(\mu)$ due to the assumption of the smoothness of $\mu$ at $\median$. Take any $y \notin H$. Again, by Theorem~\ref{theorem:ray basis depth regions} we obtain $G \in \half$ such that $\median \subset \bd{G}$, $y \in \intr{G}$, and $\mu\left(G\right) \leq \alpha^*(\mu)$. The conditions $\dim(\median) = d-1$, $\median \subset \bd{H} \cap \bd{G}$ and $H \ne G$ determine that $H$ and $G$ must be complementary, i.e. $H \cup G = \R^d$. Necessarily, $\mu\left(\R^d\right) = \mu(H) + \mu(G) \leq 2 \alpha^*(\mu)$, meaning that the measure $\mu$ is halfspace symmetric \cite{Zuo_Serfling2000c}. But, by \cite[Theorem~2.1]{Zuo_Serfling2000c} we know that the median of a halfspace symmetric measure is either unique, or $\mu$ is concentrated in an infinite line $L$ in $\R^d$. The former case contradicts $\dim\left(\median\right) = d-1$. In the latter case, also the median set of $\mu$ is contained in $L$, which contradicts our assumption of smoothness of $\mu$.

\subsection{Proof of Corollary~\ref{corollary:ray basis}}	

The first part of the statement follows directly from part~\ref{smooth case} of Theorem~\ref{theorem:ray basis depth regions} with $\alpha = \alpha^*(\mu)$. For the second part, note that due to the contiguous support of $\mu$ at $\median$, $\dim\left(\median\right) < d$ by Corollary~\ref{corollary:dimension of median}. Therefore, $\intr{\median} = \emptyset$, so $(\intr{\median})\compl=\R^d$. To obtain the result, apply Theorem~\ref{theorem:ray basis depth regions} with $x$ replaced by $y \notin \aff{\median}$. That theorem allows to choose $F$ to be $\median \in \mathfrak F\left(y,\median\right)$. We obtain that $\R^d \setminus \aff{\median}$ can be covered by halfspaces $H$ from $\half\left(\median, \median\right)$ with $\mu(H) = \alpha^*(\mu)$, but at the same time $H \in \half\left(\median, \median\right)$ implies $\median \subset \aff{\median} \subseteq \bd{H}$. In particular, we see that also $\aff{\median}$ is covered by our collection of halfspaces, and any point in the median set $x \in \median$ is contained in the boundary of $H$ for each such $H$. We conclude that $\R^d$ can be covered by halfspaces $H$ from $\half(x)$ with $\mu(H) = \D\left(x;\mu\right)$, as desired.

%
%

\subsection{Proof of Example~\ref{example:triangle gap}}	\label{section:triangle gap}

Throughout the proof we adopt the following notation. For $m = 1, 2, \dots$ and points $x_1, \dots, x_m \in \R^2$, by $x_1 \dots x_m$ we mean the polygon with vertices $x_1, \dots, x_m$. For $m = 2$ we obtain a line segment $x_1 x_2$. Where no confusion arises, by $x_1 x_2$ we also denote the length of that line segment. For polygons with non-empty interior, or more generally, for measurable sets $V \subseteq \R^2$ we denote by $\area(V)$ the two-dimensional Lebesgue measure (the area) of $V$. Recall that for $a, b \in \R^2$ we write $L[a,b]$ for the closed line segment between points $a$ and $b$ (that is, $a b$), $L(a,b) = L[a,b]\setminus\{a,b\}$ is the corresponding open line segment, and $l(a,b)$ the infinite line spanned by $a$ and $b$. 

We suppose the distances between the origin $o = \left(0,0\right)\tr$ and the points $x_c$ and $y_c$ take values in ranges $x\in(0,1/4)$ and $y\in (0,1/5)$, respectively. Set $d = \left(0,-1\right)\tr$. Because $\mu$ is a uniform distribution on $S$ with total mass $\area(S)$, the $\mu$-mass of any measurable set $V$ is proportional to $\area(S \cap V)$. For $z\in (0,y]$, $z_c=(0,-z)\tr$, and $\theta \in (-\pi/2, \pi/2]$ denote by $l_{z,\theta}$ the infinite line $l\left(z_c, z_c + \left( \cos\left( \theta - \pi/2 \right), \sin\left( \theta - \pi/2 \right) \right)\tr \right)$ that passes through $z_c$ and determines the angle $\theta$ with line $l(a,d)$. In Figure~\ref{figure:triangle gap Y} the line $l_{z,\theta}$ is drawn for $\theta = \pi/4-3/10$. Any $l_{z,\theta}$ determines two closed halfplanes. We write $H_z^{\theta+}$ for the one that contains $d$, and write $H_z^{\theta-}$ for the complementary closed halfplane. Denote 
	\[	\area_z^{\theta+} = \area(S \cap H_z^{\theta+}) \quad \mbox{ and } \quad \area_z^{\theta-} = \area(S \cap H_z^{\theta-}).	\]
By Lemma~\ref{lemma:minimizing halfspace} a minimizing halfplane of $\mu$ at $z_c$ exists. In order to determine one, we find $\theta$ that minimizes $\min \left\{\area_z^{\theta+},\area_z^{\theta-}\right\}$. The condition $\area_z^{\theta+}+\area_z^{\theta-}=\area(S)$ reduces the problem to the search for the minimum and the maximum of $\area_z^{\theta+}$. Because of the symmetry of the figure $S$, it is enough to consider $\theta\in [0,\pi/2]$.

As depicted in Figure~\ref{figure:triangle gap Y angles}, denote by $o_l$ and $o_r$ the points of intersection of the line $l\left(o,(0,1)\tr\right)$ with $L(a,b)$ and $L(a,c)$, respectively. Computing the lengths of the corresponding line segments we obtain $o o_r=2 d c/3=2/\sqrt{3}$. Using the similarity of triangles, $a x_c:a o=x x_r: o o_r$, which together with $a o=2$ and $a x_c=2-x$ gives 	
	\begin{equation*}	
	x_c x_r = x_c x_l = \frac{2-x}{\sqrt{3}}.	
	\end{equation*}
Further, we obtain $x:y=(o o_r-x_c x_r):(y_c y_r-o o_r)$, and consequently 
	\begin{equation*}	
	y_c y_r = y_c y_l = \frac{y+2}{\sqrt{3}}.
	\end{equation*}

\smallskip
\noindent \textbf{Part (I): $y_c\in \median$ and $y_c$ satisfies the covering property \eqref{ray basis}.}

We start by considering those values of $\theta \in [0,\pi/2]$ that correspond to the change of the behaviour in the formula for the area $\area_y^{\theta+}$. In the considered range for the angle $\theta$, this happens twice: \begin{enumerate*}[label=(\roman*)] \item at $\theta_1$ when the line $l_{y,\theta_1}$ passes through the vertex $c$, see the left hand panel of Figure~\ref{figure:triangle gap Y angles}; and \item at $\theta_2$ when the line $l_{y,\theta_2}$ passes through the point $x_l$, see the right hand panel of Figure~\ref{figure:triangle gap Y angles}. \end{enumerate*} Denote $e=l_{y,\theta}\cap l(a,b)$ and $f=l_{y,\theta}\cap l(b,c)$. Note that $\theta_1$ and $\theta_2$ correspond to $f=c$ and $e=x_l$, respectively. A simple computation yields 
	\begin{equation*}
	\begin{aligned}
	\tan \theta_1 & = d c/y_c d = \sqrt{3}/(1-y), \\
	\tan \theta_2 & = x_c x_l/y_c x_c = (2-x)/(\sqrt{3}(x+y)). 
	\end{aligned}
	\end{equation*}
We see that $\theta_1 \rightarrow \pi/3$ and $\theta_2 \rightarrow \pi/2$ as $x$ and $y$ both decrease to $0$. Thus, for $x$ and $y$ small enough, $\theta_1 < \theta_2$; in particular, it is straightforward to establish that for our choices of $x \in (0,1/4)$ and $y \in (0,1/5)$, we have $\theta_1 < \theta_2$. We distinguish three cases, according to the value of $\theta$, and evaluate the function $\area_y^{\theta+}$ in each situation.

\begin{figure}[htpb]
\includegraphics[width=\twofig\textwidth]{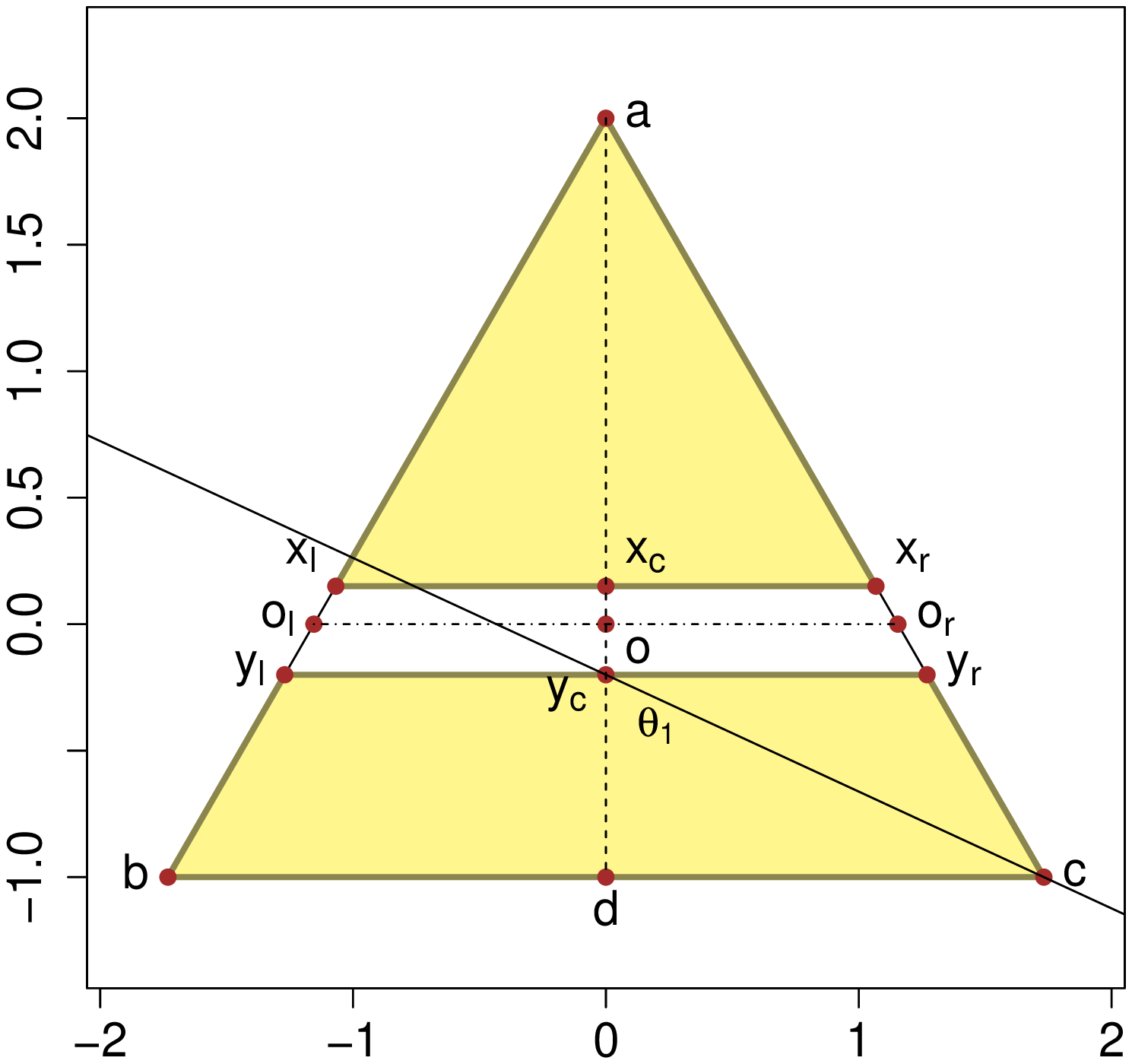} \quad
\includegraphics[width=\twofig\textwidth]{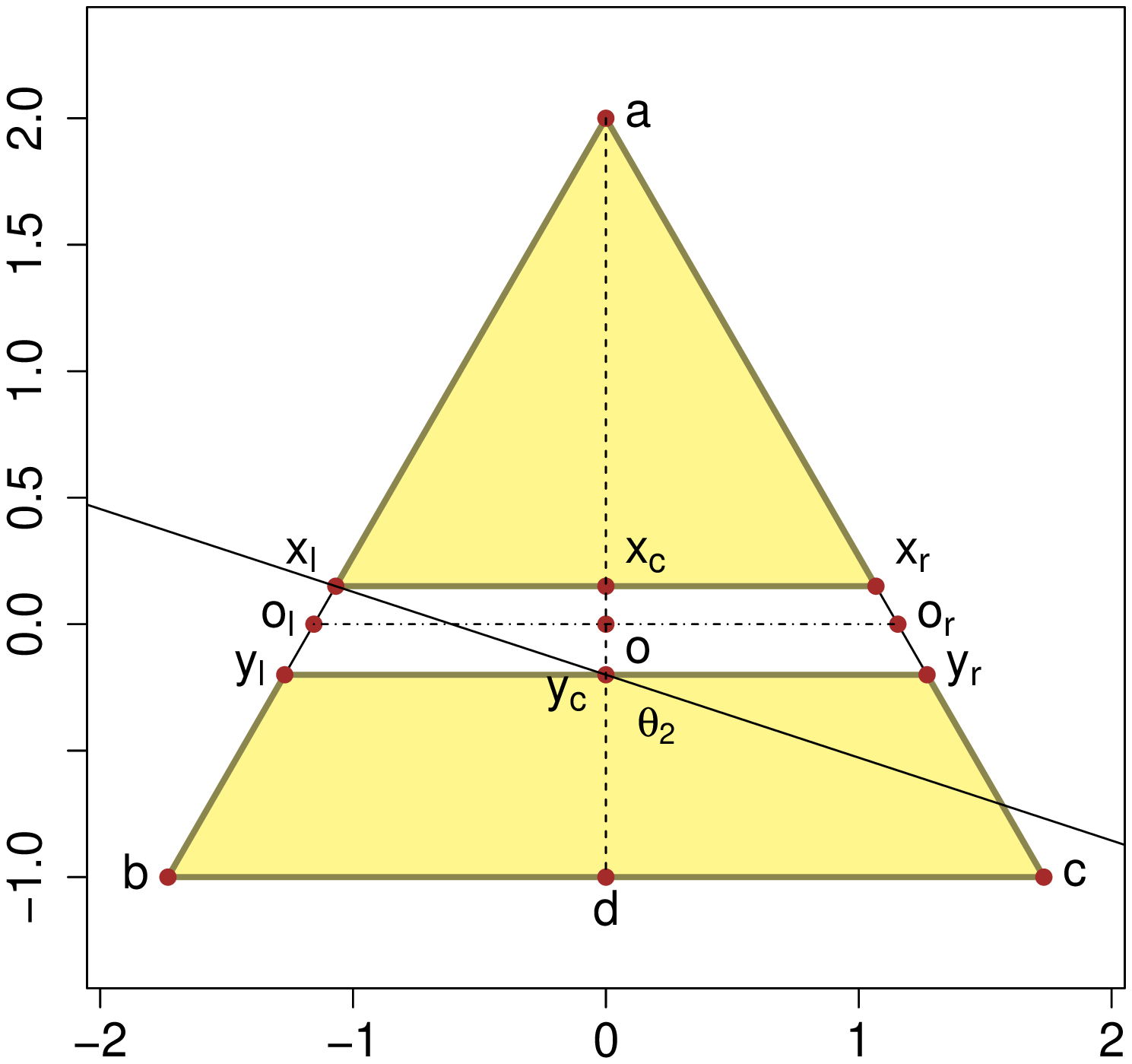}
\caption{Proof of Example~\ref{example:triangle gap}: Angles $\theta_1$ (left panel) and $\theta_2$ (right panel) represent the two values in the interval $[0,\pi/2]$ where the function $\area_y^{\theta+}$ fails to be differentiable.}
\label{figure:triangle gap Y angles}
\end{figure}

\textbf{Case (i): $0 \leq \theta < \theta_1$.}

\begin{figure}[htpb]
\includegraphics[width=\twofig\textwidth]{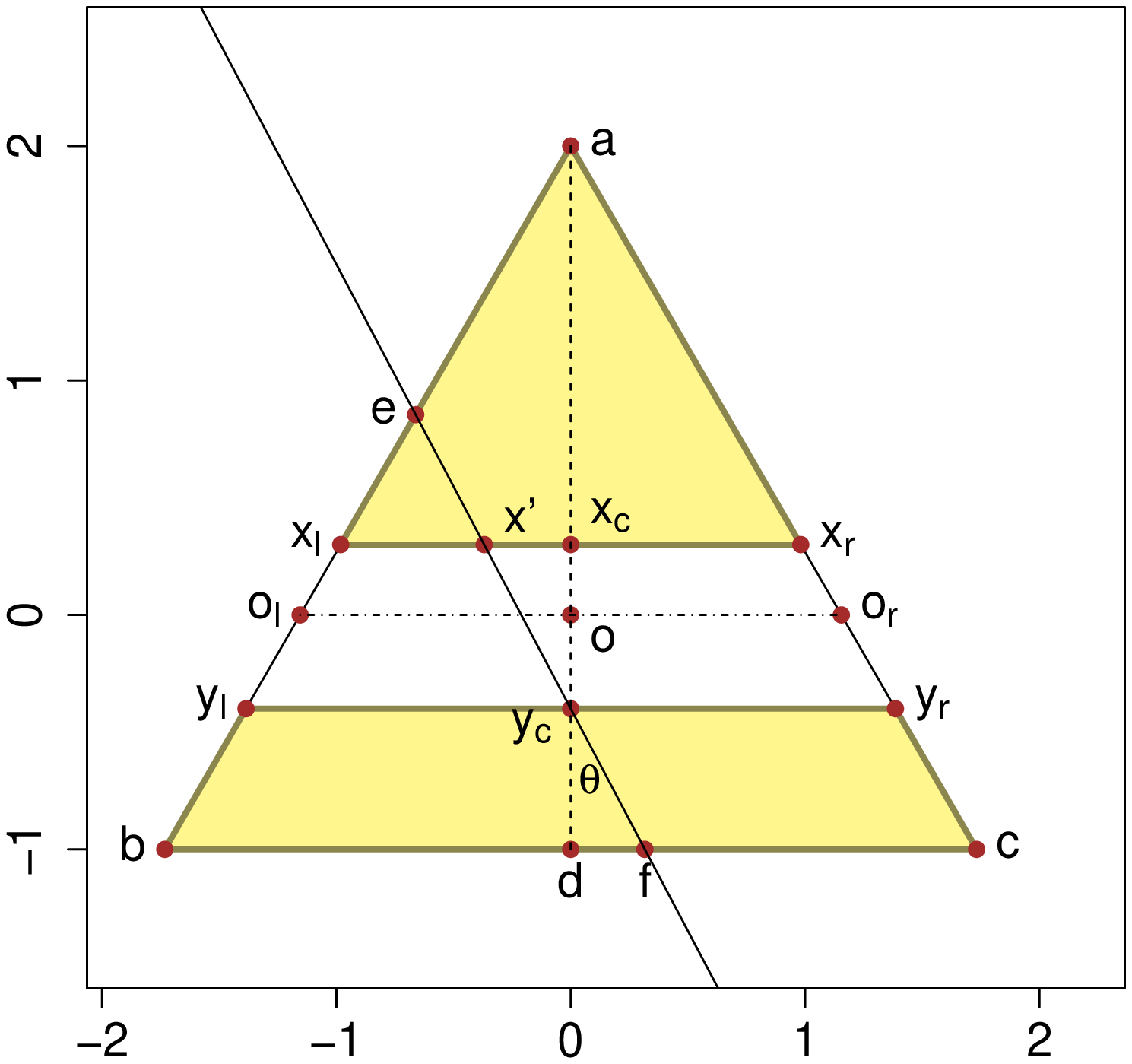} \quad
\includegraphics[width=\twofig\textwidth]{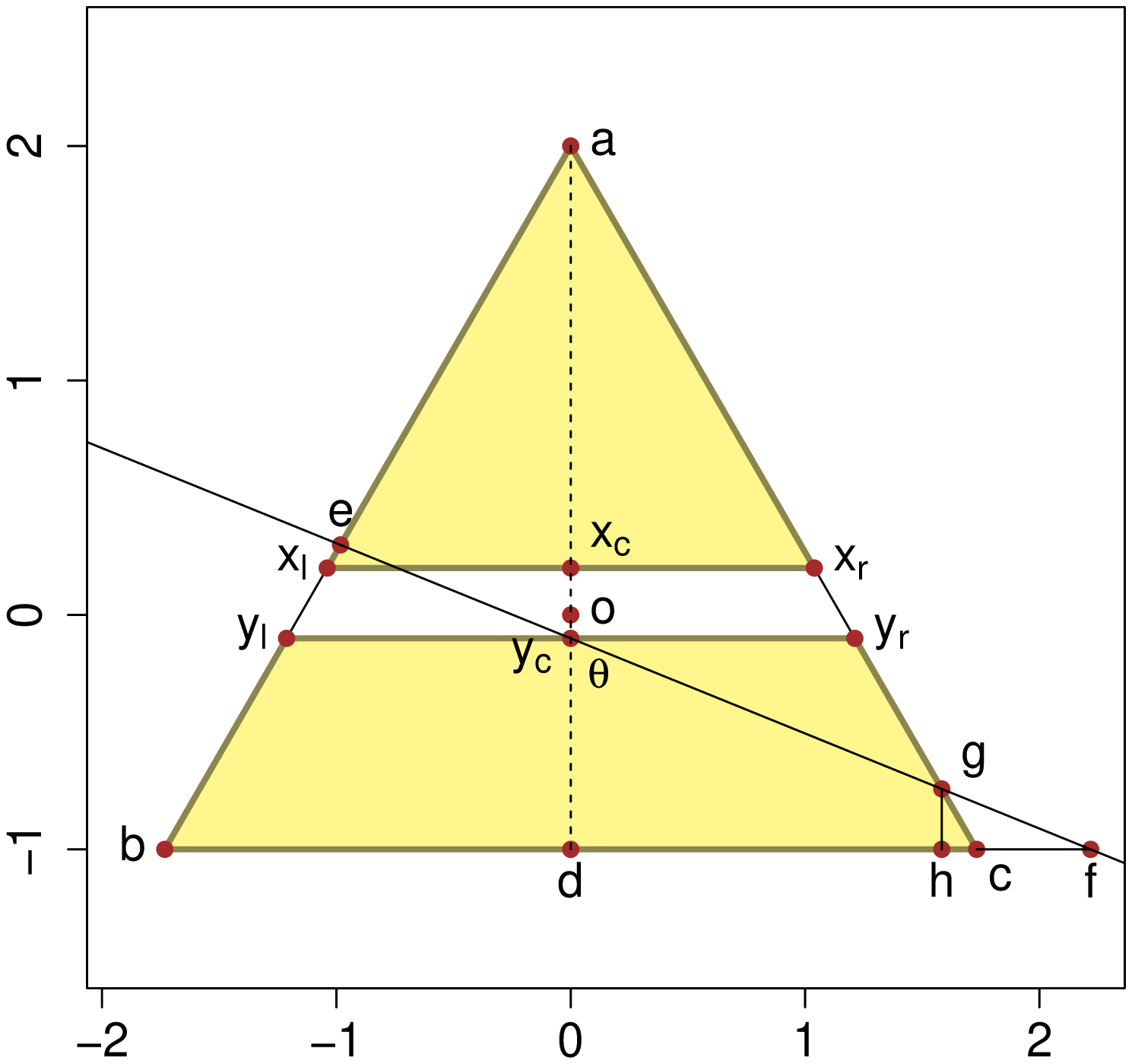}
\caption{Proof of Example~\ref{example:triangle gap}: The situation that corresponds to the Case (i) of Part (I) of the proof (left panel), and Case (ii) of Part (I) of the proof (right panel).}
\label{figure:triangle gap Y}
\end{figure}

For $x'$ the point of intersection of $L(e,f)$ and $L(x_l,x_r)$, we derive $\area_y^{\theta+}$ by computing  
	\begin{equation}\label{surface difference}
	\area_y^{\theta+}=\area(e b f)-\area(x' x_l y_l y_c).
	\end{equation}

Denote $\area_x=\area(o o_r x_r x_c)$ and $\area_y=\area(y_c y_r o_r o)$. Then $\area_x=(x_c x_r + o o_r) \cdot x/2=x (4-x)/(2\sqrt{3})$ and $\area_y=(y_c y_r+o o_r) \cdot y/2=y (4+y)/(2\sqrt{3})$. Denote by $h$ the length of the normal from $e$ to line $l(y_l,y_c)$. We compute $h:(h+1-y)=y_c y_l:b f$. On the other hand, $b f=b d+ d f=\sqrt{3}+(y_c d) \tan{\theta}=\sqrt{3}+(1-y)\tan{\theta}$. From here, we get $h=(y+2)/(1+\sqrt{3}\tan{\theta})$, and consequently 
	\begin{equation}\label{big triangle}
	\area(e b f)=(h+1-y)\frac{b f}{2}=\frac{(\sqrt{3}+(1-y)\tan{\theta})^2}{2(\frac{1}{\sqrt{3}}+\tan{\theta})}.
	\end{equation}

Since $x_c y_c =x+y$ and the angle between $l(y_c,x')$ and $l(y_c,x_c)$ equals $\theta$, it holds true that $\area(x_c x' y_c)=(x+y)^2 \tan{\theta}/2$ and
	\begin{equation}\label{gap}
	\area(x' x_l y_l y_c)=\area_x+\area_y-\area(x_c x' y_c)=\frac{(4-x)x}{2\sqrt{3}}+\frac{(4+y)y}{2\sqrt{3}}-\frac{\tan{\theta}}{2}(x+y)^2.
	\end{equation}

Finally, substituting~\eqref{big triangle} and~\eqref{gap} into~\eqref{surface difference} gives
	\begin{equation*}
	\area_y^{\theta+}=\frac{(\sqrt{3}+(1-y)\tan{\theta})^2}{2(\frac{1}{\sqrt{3}}+\tan{\theta})}-\frac{2}{\sqrt{3}}(x+y)-\frac{1}{2\sqrt{3}}(y^2-x^2)+\frac{\tan{\theta}}{2}(x+y)^2.
	\end{equation*}
Now we take $t={1}/{\sqrt{3}}+\tan{\theta}$, and substitute $\tan{\theta}$ by $t-1/{\sqrt{3}}$ in the last equation. Consolidating the terms with $t$, we get
	\begin{equation}	\label{area theta}
	\area_y^{\theta+}=\frac{1}{\sqrt{3}}(2+y)(1 - x - 2y)+F_1(t)
	\end{equation}
where 
	\begin{equation*}
	F_1(t)=\frac{\left(2+y\right)^2}{6}\frac{1}{t}+\left((1-y)^2+(x+y)^2\right)\frac{t}{2}.
	\end{equation*}
The derivative of $F(t)$ is given by 
	\begin{equation*}
	F_1'(t)=\frac{(1-y)^2+(x+y)^2}{2}-\frac{\left(2+y\right)^2}{6}\frac{1}{t^2}.
	\end{equation*}
Condition $F_1'(t)=0$ gives $t^2={\left(2+y\right)^2}/{\left(3\left((1-y)^2+(x+y)^2\right)\right)}$. Denote 
	\begin{equation*}
	t_m=\frac{2+y}{\sqrt{3\left((1-y)^2+(x+y)^2\right)}}.
	\end{equation*}
The derivative $F_1'(t)$ is negative for $t<t_m$ and positive for $t>t_m$. Therefore, the function $F_1(t)$ decreases with $t$ until $t=t_m$, and then increases again, so it is only left to check whether $t_m<\tan{\theta_1}+1/\sqrt{3}$. In order to do so, we should verify
	\begin{equation*}
	\frac{2+y}{\sqrt{3\left((1-y)^2+(x+y)^2\right)}}<\frac{\sqrt{3}}{1-y}+\frac{1}{\sqrt{3}}=\frac{4-y}{\sqrt{3}(1-y)},
	\end{equation*}
which can be rewritten as
	\begin{equation}\label{condition for minimum theta}
	\frac{2+y}{\sqrt{1+\left(\frac{x+y}{1-y}\right)^2}}<4-y.
	\end{equation}
Since the left hand side of the last equation is smaller than $2+y$, it is enough to show $2+y<4-y$, which is satisfied because we started with $y<1/5$. We conclude that $F_1(t)$ attains its minimum value at $t=t_m$, or equivalently, at $\tan{\theta_{\min}}=t_m-1/\sqrt{3}$. The appropriate minimum area is given by
	\begin{equation} \label{minimum area case i}
	\area_{y}^{\theta_{\min}+} = \frac{1}{\sqrt{3}}(2+y)(1 - x - 2y)+\frac{1}{\sqrt{3}}(2+y)\sqrt{(1-y)^2+(x+y)^2}.
	\end{equation}

\textbf{Case (ii): $\theta_1 \leq \theta < \theta_2$.}


Now $e\in L[a,x_l]$ and $f\notin L(d,c)$, see the right hand panel of Figure~\ref{figure:triangle gap Y}. Denote $g=L[y_c,f] \cap L[a,c]$. In order to get the formula for $\area_y^{\theta+}$ in this case, one simply subtracts $\area(g c f)$ from formula~\eqref{area theta}. Denote by $h$ the orthogonal projection of $g$ onto $l(d,f)$. Then $\tan\theta=f d/y_c d=(f c+\sqrt{3})/(1-y)$, so $f c=(1-y)\tan\theta-\sqrt{3}$, and $h c=h g/\sqrt{3}$. At the same time, $\tan\theta=h f/h g=(h c + c f)/h g=c f/h g+1/\sqrt{3}$, implying $h g=c f/(\tan\theta-1/\sqrt{3})$. Now we may calculate $\area(g c f)$ as $c f\cdot h g/2$, which, after the substitution $t= 1/\sqrt{3} + \tan \theta$, gives
	\begin{equation*}
	\area(g c f)=\frac{(1-y)^2}{2}t+\frac{(2+y)^2}{6}\frac{1}{t-\frac{2}{\sqrt{3}}}-\sqrt{3}(1-y).
	\end{equation*}
For $\theta=\theta_1$ one gets $\area(g c f)=0$, as expected. Finally, we derive 
	\begin{equation}	\label{area case ii}
	\area_y^{\theta+}=\frac{1}{\sqrt{3}}(5-6y-2y^2-2x-x y)+F_2(t),
	\end{equation}
where
	\begin{equation*}
	F_2(t)=(x+y)^2\frac{t}{2}+\frac{\left(2+y\right)^2}{6}\left(\frac{1}{t}-\frac{1}{t-\frac{2}{\sqrt{3}}}\right).
	\end{equation*}
The above formula holds true for $\tan \theta_1+1/\sqrt{3}\leq t < \tan\theta_2+1/\sqrt{3}$. By introducing $s=t-\tan\theta_1-1/\sqrt{3}\geq 0$, we may write 

	\begin{equation} \label{analogy1}
	F_2(t)=\frac{(x+y)^2}{2\sqrt{3}}\frac{4-y}{1-y}+G(s),
	\end{equation}
where 
	\begin{equation}	\label{analogy2}
	G(s)=\frac{(x+y)^2}{2}s-\frac{\left(2+y\right)^2}{3\sqrt{3}}\frac{1}{(s+\frac{\sqrt{3}}{1-y})^2-\frac{1}{3}}.	
	\end{equation} 
One may investigate $G(s)$ instead of $F_2(t)$, because $t\mapsto s$ is a simple shift by a constant. The derivative of $G(s)$,
	\begin{equation}\label{derivative of substitute}
	G'(s)=\frac{(x+y)^2}{2}+\frac{2(2+y)^2}{3\sqrt{3}}\frac{1}{((s+\frac{\sqrt{3}}{1-y})^2-\frac{1}{3})^2}\left(s+\frac{\sqrt{3}}{1-y}\right),
	\end{equation}
is always positive for $s\geq 0$, so $\area_y^{\theta+}$ grows with $\theta$ for $\theta_1 \leq \theta < \theta_1$.

\textbf{Case (iii): $\theta_1 \leq \theta \leq \pi/2$.}

As $e\in L\left[x_l,y_l\right]$, $\area_y^{\theta+}$ obviously grows with increasing $\theta$.

\smallskip
\noindent\textbf{Summary of Part (I).}

Now that we established the behaviour of the mapping $\theta\mapsto \area_y^{\theta+}$, we are able to find its extreme values. The value of $\area_y^{\theta+}$ decreases for $0 \leq \theta < \theta_{\min}$ and increases for $\theta_{\min} < \theta \leq \pi/2$, so it attains minimum value at $\theta=\theta_{\min}$, and maximum value at $\theta=0$ or $\theta=\pi/2$. Therefore, the depth of the point $y_c$ with respect to $\mu$ is equal to $\min \left\{\area_{y}^{\theta_{\min}+}, \area_{y}^{0-}, \area_{y}^{\pi/2-}\right\}$. Note that $\area_{y}^{0-} = \lambda(S)/2$ and $\area_{y}^{\pi/2-}= \area(a x_l x_r)$. As $x,y\rightarrow 0$ from the right, $\area(a x_l x_r)$ goes to $4/9 \cdot \area(a b c)$, and $\area(S)$ goes to $\area(a b c)$. In particular, it is easy to show that for $x \in (0,1/4)$ and $y\in(0,1/5)$ we have $\area_{y}^{\pi/2-} < \area_{y}^{0-}$. In what follows we show that it is possible to choose $x$ and $y$ such that $\area_{y}^{\pi/2-}=\area_{y}^{\theta_{\min}+}$.

Note that $\area_{y}^{\pi/2-}=\area(a x_l x_r)=x_c x_r\cdot a x_c=(2-x)^2/\sqrt{3}$. Therefore, the condition $\area_{y}^{\pi/2-}=\area_{y}^{\theta_{\min}+}$ is, using \eqref{minimum area case i}, equivalent with
	\begin{equation*}
	f(x,y)=-2-3y-2y^2+2x-x^2-x y+(2+y)\sqrt{(1-y)^2+(x+y)^2}=0.
	\end{equation*}
Consider $x \in (0,1/4)$ fixed. As $y\rightarrow 0$, $f(x,y)$ converges to $f_0(x)=(\sqrt{2x}+\sqrt{1+x^2}-1)(\sqrt{2x}+1-\sqrt{1+x^2})$. To show that $f_0(x)>0$ for any choice of $x \in (0,1/4)$, we should prove $\sqrt{2x}+1>\sqrt{1+x^2}$, which is equivalent with $2x+2\sqrt{2x}-x^2>0$. The latter inequality follows from $2x+2\sqrt{2x}-x^2=2\sqrt{2x}+x(2-x)>0$. In the case $y\rightarrow 1/5$, $f(x,y)$ converges to $f_{1/5}(x)=-67/25 + 9x/5 - x^2 + 11 \sqrt{15/25+\left(x + 1/5\right)^2}/5 \leq -67/25 + 9x/5 + 11 \sqrt{15/25+\left(x + 1/5\right)^2}/5 < 0$ for each $x \in (0,1/4)$. Because the mapping $y\mapsto f(x,y)$ is continuous in $y$ for each fixed $x$, the intermediate value theorem \cite[Problem~2.2.14(d)]{Dudley2002} allows us to conclude the existence of $y \in (0,1/5)$ such that $f(x,y)=0$ for each $x \in (0,1/4)$.

Finally, from all the considerations above, we conclude that for each $x<1/4$, there exists $y$ such that $\D(y_c;\mu)=\area(a x_l x_r)$ and there are three minimizing halfspaces of $\mu$ at $y_c$, being $H_{y}^{\theta_{\min}+}$, $H_{y}^{-\theta_{\min}+}$ and $H_{y}^{\pi/2-}$. Obviously, those halfspaces cover the whole $\R^2$, so $y_c$ satisfies the covering property \eqref{ray basis}, and by part~\ref{direct ray basis} of Theorem~\ref{theorem:standard ray basis}, $y_c$ must be a median of $\mu$. The course of the function $\area_{y}^{\theta+}$ can be observed in Figure~\ref{figure:lambda functions}.

\begin{figure}[htpb]
\includegraphics[width=\twofig\textwidth]{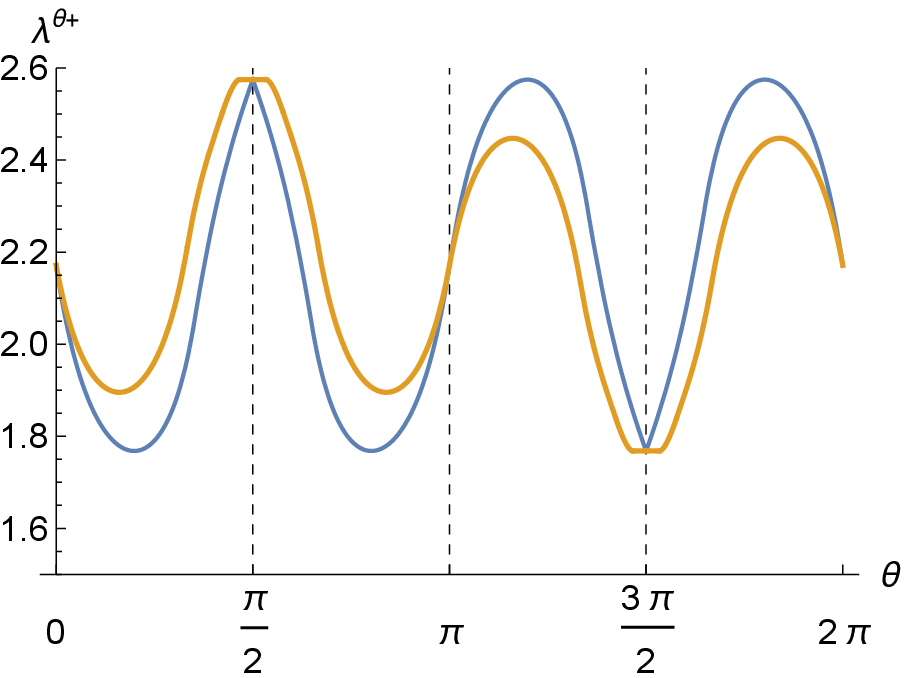} \quad
\includegraphics[width=\twofig\textwidth]{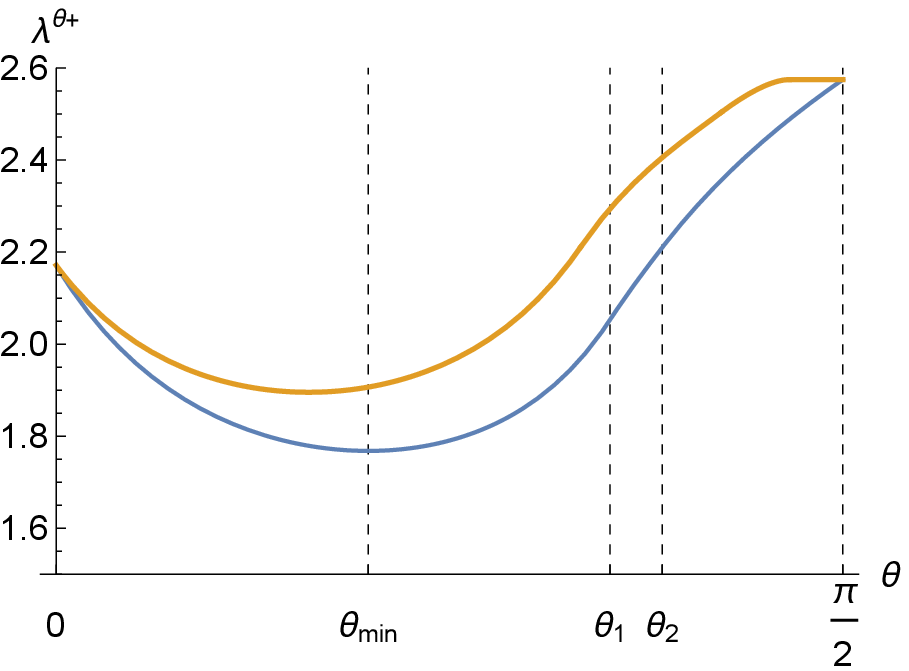}
\caption{Example~\ref{example:triangle gap}: Functions $\area_y^{\theta+}$ (blue curve) and $\area_0^{\theta+}$ (thick orange curve) on their full domain $[0,2\pi]$ (left panel) and the same functions zoomed into the interval $[0,\pi/2]$ (right panel). Both functions attain the same minimum value. On the right hand panel, the point $\theta_{\min}$ where $\area_y^{\theta+}$ attains its minimum, along with the points $\theta_1$ and $\theta_2$ of non-differentiability of function $\area_y^{\theta+}$ are marked by dashed vertical lines.}
\label{figure:lambda functions}
\end{figure}

\smallskip
\noindent\textbf{Part (II): $\dim\left(\median\right) = 2$.}

We first show that there exists a point $z_c = \left(0, -z\right) \tr \ne y_c$ in $\R^2$ that is also a halfspace median of $\mu$. For that reason, we use the results from Part (I) to calculate $\area_z^{\theta+}$. We consider only $x$ and $y$ from Part (I) that satisfy $\area_{y}^{\pi/2-}=\area_{y}^{\theta_{\min}+}=\D(y_c;\mu)=\alpha^*(\mu)$. 

Denote $q = x_r y_l \cap x_l y_r$ and consider point $z_c=(0,-z)\tr$ for $z\in (0,y)$ on the line segment $L(q, y_c)$, see the left panel of Figure~\ref{figure:triangle gap Z}. Because $z_c \in L(q,y_c)$, the line $l(x_l,z_c)$ intersects $L[y_c,y_r]$. Denote $\delta=y-z>0$ and $y' = l(x_l, z_c) \cap L[y_c,y_r]$.

\begin{figure}[htpb]
\includegraphics[width=\twofig\textwidth]{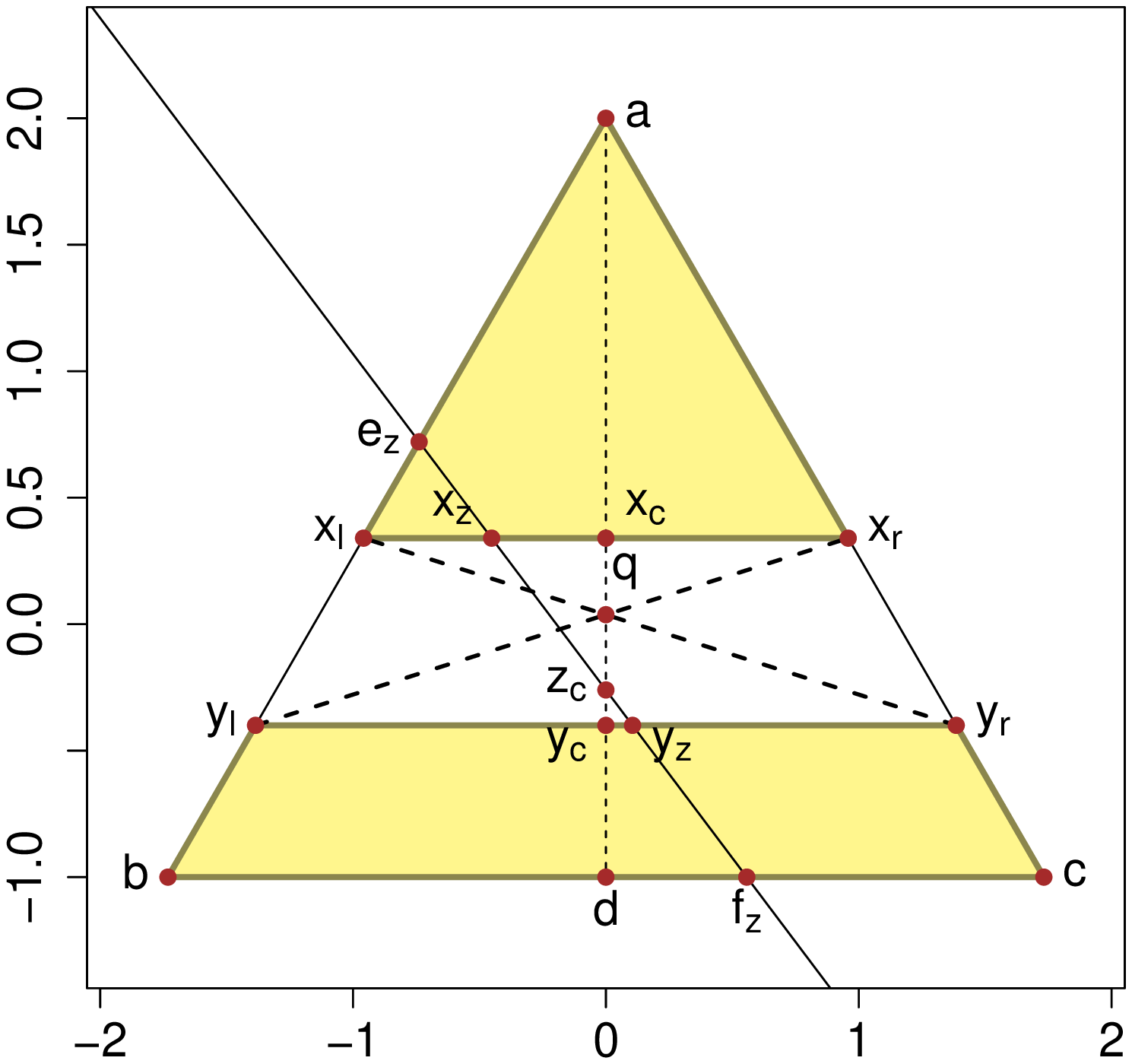} \quad
\includegraphics[width=\twofig\textwidth]{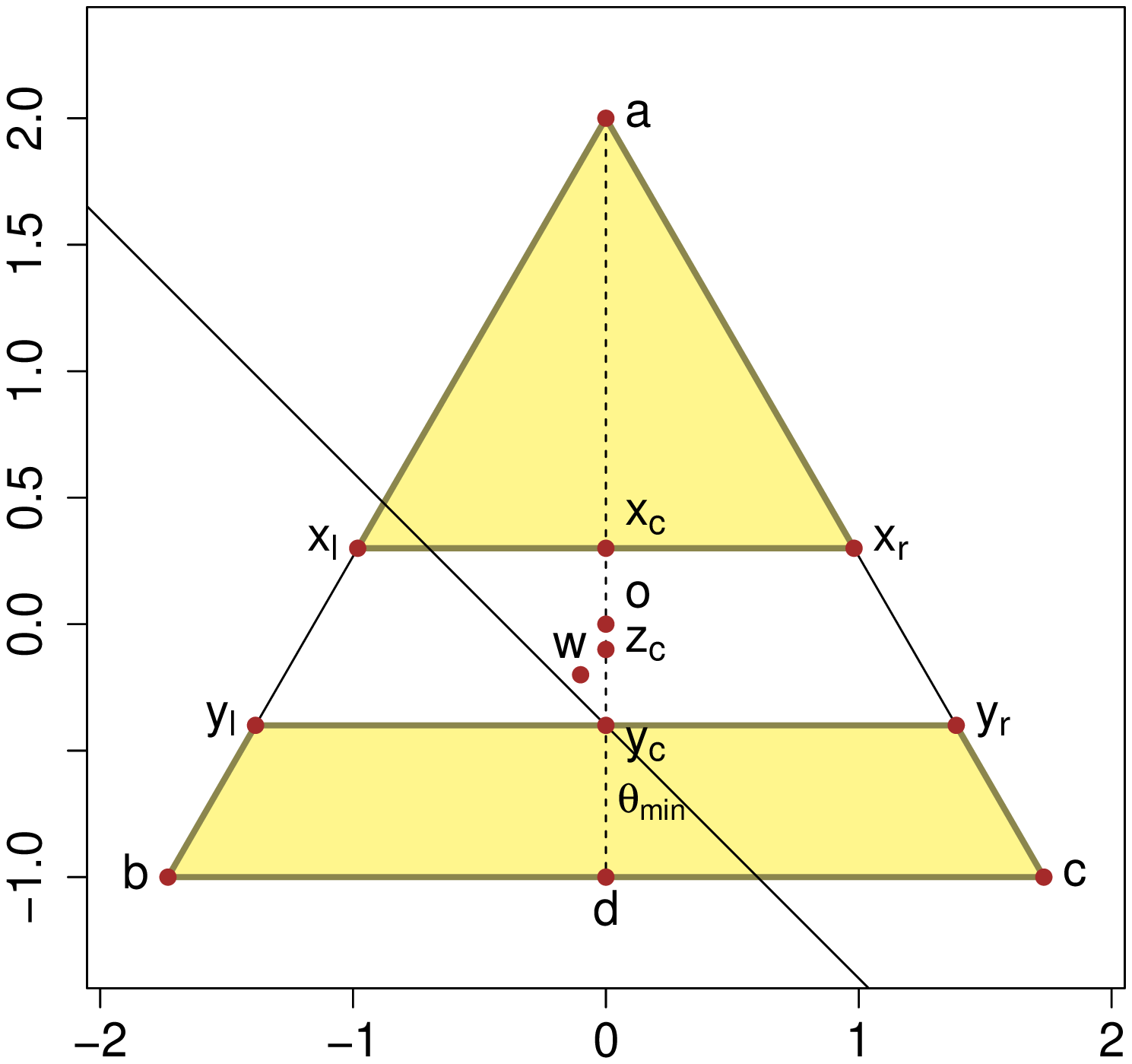}
\caption{Proof of Example~\ref{example:triangle gap}: A scheme of the configuration that corresponds to Part (II) of the proof, with point $z_c \in \median$ displayed (left panel), and the region where the additional median $w$ is searched for (right panel).}
\label{figure:triangle gap Z}
\end{figure}

Denote $x_z=l_{z,\theta}\cap l(x_l,x_r)$, $y_z=l_{z,\theta}\cap l(y_l,y_r)$, $e_z=l_{z,\theta}\cap l(a,b)$ and $f_z=l_{z,\theta}\cap l(b,c)$. Consider angles $\widetilde{\theta_1}$, $\widetilde{\theta_2}$ and $\widetilde{\theta_3}$ corresponding to cases $f_z=c$, $e_z=x_l$ and $y_z=y_r$, respectively. It is not difficult to calculate 
	\begin{equation*}
	\begin{aligned}
	\tan \widetilde{\theta_1}& = d c/z_c d = \sqrt{3}/(1-z), \\
	\tan \widetilde{\theta_2}& = x_c x_l/z_c x_c = (2-x)/(\sqrt{3}(x+z)), \\
	\tan \widetilde{\theta_3}& = y_c y_r/y_c z_c = (y+2)(\sqrt{3}\delta). 
	\end{aligned}
	\end{equation*}
If $\delta$ is small enough, then $\widetilde{\theta_1}<\widetilde{\theta_2}<\widetilde{\theta_3}$. We consider four different cases and calculate $\area_z^{\theta+}$.

For $0\leq \theta < \widetilde{\theta_1}$, one obtains $\area_z^{\theta+}$ by substituting $y$ by $z$ in \eqref{area theta}, and subsequently subtracting $\area(y_l y_z z_c z_l)$, where $z_l = l_{z,\pi/2} \cap l(a,b)$. Note that $\area(z_c y_c y_z)= \delta^2 \tan\theta/2$ and $\area(y_l y_c z_c z_l)=(4+y+z)\delta/(2\sqrt{3})$, so 
	\begin{equation}	\label{area gap}
	\area(y_l y_z z_c z_l) = \frac{\delta^2}{2}t+\frac{4+y+z}{2\sqrt{3}}\delta-\frac{1}{2\sqrt{3}}\delta^2,	
	\end{equation}
where $t=1/\sqrt{3} + \tan\theta$.

\textbf{Case (i): $0 \leq \theta < \widetilde{\theta_1}$.}

The previous calculation leads to
	\begin{equation*}
	\area_z^{\theta+}=\frac{1}{\sqrt{3}}\left(2-3z-2z^2-2x-x z-2\delta-\frac{y+z}{2}\delta+\frac{\delta^2}{2}\right)+\widetilde{F}_1(t),
	\end{equation*}
where
	\begin{equation*}
	\widetilde{F}_1(t)=\frac{\left(2+z\right)^2}{6}\frac{1}{t}+\left((1-z)^2+(x+z)^2-\delta^2 \right)\frac{t}{2}.
	\end{equation*}
For $\delta$ small, $\widetilde{F}_1(t)$ behaves similarly to $F_1(t)$. In that case $\area_z^{\theta+}$ decreases with $\theta$ for small positive values of $\theta$, reaches its minimum value at
	\begin{equation*}
	\tan \widetilde{\theta}_{\min}=\frac{2+z}{\sqrt{3\left((1-z)^2+(x+z)^2-\delta^2\right)}}-\frac{1}{\sqrt{3}},
	\end{equation*}
and then increases again as $\theta$ grows. We also check $\tan \widetilde{\theta}_{\min}<\tan\widetilde{\theta_1}$ in the same way as we did in Part~(I) of this proof --- analogously to \eqref{condition for minimum theta}, we should verify 
	\[	\frac{2+z}{\sqrt{1+\left(\frac{x+z}{1-z}\right)^2-\frac{\delta^2}{(1-z)^2}}}<4-z.	\] 
For $\delta$ positive small enough the left hand side of the previous display is smaller than $2+z$, and we may conclude $\widetilde{\theta}_{\min}<\widetilde{\theta_1}$.

\textbf{Case (ii): $\widetilde{\theta_1} \leq \theta < \widetilde{\theta_2}$.}

In this situation, analogously to Case (i) we substitute $y$ by $z$ in \eqref{area case ii} and subtract \eqref{area gap} to obtain
	\begin{equation*}
	\area_z^{\theta+}=\frac{1}{\sqrt{3}}(5-6z-2z^2-2x-x z-2\delta-\frac{z+y}{2}\delta+\frac{\delta^2}{2})+\widetilde{F}_2(t),
	\end{equation*}
where
	\begin{equation*}
	\widetilde{F}_2(t)=\left((x+z)^2-\delta^2\right)\frac{t}{2}+\frac{\left(2+z\right)^2}{6}\left(\frac{1}{t}-\frac{1}{t-\frac{2}{\sqrt{3}}}\right).
	\end{equation*}
We introduce again $s=t-\tan\widetilde{\theta_1}-1/\sqrt{3} \geq 0$ and investigate $\widetilde{G}(s)$, analogously as in \eqref{analogy1} and \eqref{analogy2}, instead of $\widetilde{F}_2(t)$. As in \eqref{derivative of substitute}, we obtain the derivative of $\widetilde{G}(s)$ of the form
	\begin{equation*}
	\widetilde{G}'(s)=\frac{(x+z)^2 - \delta^2}{2}+\frac{2(2+z)^2}{3\sqrt{3}} \left(\left(s+\frac{\sqrt{3}}{1-z}\right)^2-\frac{1}{3}\right)^{-2}\left(s+\frac{\sqrt{3}}{1-z}\right), 
	\end{equation*}
which is positive for small $\delta$, meaning that $\area_z^{\theta+}$ increases as a function of $\theta \in \left[ \widetilde{\theta_1}, \widetilde{\theta_2} \right)$.

\textbf{Case (iii): $\widetilde{\theta_2} \leq \theta < \widetilde{\theta_3}$.}

For these values of the parameter, it is easy to see that $\area_{z}^{\theta+}$ grows with increasing $\theta$, for the same reason as in Part (I) of this proof.

\textbf{Case (iv): $\widetilde{\theta_3} \leq \theta \leq \frac{\pi}{2}$.}

In this range for $\theta$, the function $\area_z^{\theta+}$ is clearly constant, and equal to $\area(y_r y_l b c)$.

\smallskip
\noindent\textbf{Summary of Part (II).}

We conclude our analysis by observing that $\area_z^{\theta+}$ attains its minimum at $\theta=\widetilde{\theta}_{\min}$, and its maximum at $\theta=\pi/2$. Because $\widetilde{\theta}_{\min} \in [0, \widetilde{\theta_1})$, $y_c$ lies in the interior of $H_{z}^{\widetilde{\theta}_{\min}+}$ and therefore $\area_{z}^{\widetilde{\theta}_{\min}+} > \area_{y}^{\theta_{\min}+}=\area_{y}^{\pi/2-} = \alpha^*(\mu)$. On the other hand, $\area_{z}^{\pi/2-}=\area_{y}^{\pi/2-}=\alpha^*(\mu)$. Because by Part (I) we know that $y_c$ is a median of $\mu$, this means that $z_c$ and $y_c$ have the same depth, and the set $\median$ contains at least two distinct points $z_c$ and $y_c$. The minimizing halfspaces of $\mu$ at $z_c$ are those determined by the angle $\theta \in [-\pi/2,-\widetilde{\theta_3}] \cup [\widetilde{\theta_3},\pi/2]$. Therefore, it is impossible to cover $\R^2$ with minimizing halfspaces of $\mu$ at point $z_c$, meaning that $z_c\in \median$ is a point that fails to satisfy the covering property \eqref{ray basis}.

Due to the convexity of the median set $\median$ we already know that the line segment $L[y_c,z_c]$ is contained in $\median$. To see that the set $\median$ is in fact two-dimensional, now we find an additional median of $\mu$ of the form $w=(-\varepsilon,\tau)\tr \in \left(\left(-\infty,0\right) \times (-y, -z) \right) \setminus H_{y}^{\theta_{\min}+}$, see the right hand panel of Figure~\ref{figure:triangle gap Z}. The smoothness of $\mu$ implies that for $\tau$ fixed, $\inf \left\{\mu(H) \colon H \in \half(w) \mbox{ and } \{y_c,z_c\}\subset H\compl \right\} \to 1/2$ as $\varepsilon \to 0$. Since $\D\left(w;\mu\right) \leq \alpha^*(\mu)<1/2$ we obtain that for any $\tau$ there exists $\varepsilon$ small enough so that each minimizing halfspace of $\mu$ at $w$ has to contain at least one of the points $y_c$ and $z_c$. Each halfspace $H \in \half(w)$ whose boundary passes through $w$ that contains either $y_c$ or $z_c$, however, must have $\mu$-mass at least $\alpha^*(\mu)$, since both $y_c$ and $z_c$ belong to $\median$. Therefore, for any $\tau \in (-y,-z)$ and all $\varepsilon$ small enough, we conclude that the point $w = \left( - \varepsilon, \tau \right)\tr$ is contained in the median set $\median$, and the latter set must therefore be two-dimensional.

\smallskip
\noindent \textbf{Part (III): The single point that satisfies the covering property \eqref{ray basis} is $y_c$.}

In Part (I) of this proof we demonstrated that $y_c$ satisfies \eqref{ray basis}. The fact that no other point in $\R^2$ shares that property is a consequence of Theorem~\ref{unique covering median} proved in Section~\ref{section:covering median}, whose assumptions are satisfied by $y_c$.

\subsection{Proof of Theorem~\ref{theorem:intersection}}
First note that all the considered sets are convex, meaning that each of them is of dimension $d$ if and only if its interior is non-empty \cite[Theorem~1.1.13]{Schneider2014}. The equality $\Ucamu = \intr{\Uamu}$ follows from Theorem~\ref{theorem:ray basis additional}; the interior of any set is a subset of its closure, which implies $\intr{\Uamu} \subseteq \cl{\Uamu}$; the floating body $\UaFBmu$ is an intersection of a larger collection of halfspaces than the depth region $\Damu$, meaning that $\UaFBmu \subseteq \Damu$. We prove the remaining non-trivial statements of the proposition in several steps. 

\textbf{Inclusion $\cl{\Uamu}\subseteq \UaFBmu$.} Suppose that there exists $x\in \Uamu \setminus \UaFBmu$. Since $x\notin \UaFBmu$, there must exist $H_{y,v} \in \half$ such that $x\notin H_{y,v}$ and $ \mu(H_{y,v}\compl)\leq \alpha$. Its complement $H_{y,v}\compl$ is an open halfspace whose boundary passes through $y$ and has inner normal $-v$, i.e. $H_{y,v}\compl=\intr{H_{y,-v}}$. We know that $x\in \intr{H_{y,-v}}$ and $\mu(\intr{H_{y,-v}})=\mu(H_{y,v}\compl)\leq \alpha$. Therefore, $H_{x,-v}\subset \intr{H_{y,-v}}$ and consequently $\mu(H_{x,-v})\leq \mu(\intr{H_{y,-v}})\leq  \alpha$, which is in contradiction with $\D(x;\mu) > \alpha$, as well as $x\in \Uamu$. We conclude that $\Uamu\subseteq \UaFBmu$ and consequently $\cl{ \Uamu}\subseteq \UaFBmu$, because $\UaFBmu$ is closed.

\textbf{Part~\ref{inclusions null}: Inclusion $\Ucamu\subseteq\left\{x\in\R^d\colon \D^\circ(x;\mu)>\alpha\right\}$.} Pick $x \in \Ucamu$, and suppose that $\D^\circ\left(x;\mu\right) = \beta \leq \alpha$. Analogously to Lemma~\ref{lemma:minimizing halfspace} it is possible to show that if $\D^\circ\left(x;\mu\right) = \beta$, there must exist $G \in \half(x)$ such that $\mu\left(\intr{G}\right) = \beta$. Define $H = \left(\intr{G}\right)\compl \in \half$ and note that $x \in \bd{H}$ and $\mu\left(H\compl\right) = \mu(\intr{G}) = \beta \leq \alpha$. We reached a contradiction, since for such a halfspace $H$ the point $x$ must be contained in $\intr{H}$.

\textbf{Part~\ref{inclusions null}: Inclusion $\left\{x\in\R^d\colon \D^\circ(x;\mu)>\alpha\right\}\subseteq\Ucamu$.} If $\D^\circ(x;\mu)>\alpha$, then for each $H\in \half$ that contains $x$ it has to be $\mu(\intr{H})>\alpha$ and therefore $x\notin \left(\Ucamu\right)\compl$.

\textbf{Part~\ref{inclusions first}: Inclusion $\UaFBmu \subseteq \cl{\Uamu}$ if $\intr{\Uamu}\neq \emptyset$.} Note that $\intr{\Uamu}\subseteq \cl{\Uamu}\subseteq \UaFBmu$, meaning that $\intr{\UaFBmu}\neq \emptyset$. Therefore, if $\UaFBmu \setminus \cl{\Uamu}$ is non-empty, there must exist $x\in \intr{\UaFBmu}\setminus \cl{\Uamu}$. To see this, denote by $y$ any point in $\UaFBmu \setminus \cl{\Uamu}$. Since $\UaFBmu$ is full-dimensional and convex, there exists an open ball $B_1 \subset \UaFBmu$ \cite[Theorem~1.1.13]{Schneider2014}, and $\conv{B_1\cup\{y\}} \subset \UaFBmu$. Since $\cl{\Uamu}$ is a closed set that does not contain $y$, there exists an open ball $B_2$ containing $y$ that does not intersect $\cl{\Uamu}$. The set $\conv{B_1\cup\{y\}} \cap B_2$ is then a convex, full-dimensional subset of $\UaFBmu \setminus \cl{\Uamu}$, and as such has to contain a point in its interior $x \in \intr{\UaFBmu} \setminus \cl{\Uamu}$, as we needed to show. Now, because $x\notin \Uamu$, Lemma~\ref{lemma:minimizing halfspace} gives that there exists $H \in \half(x)$ such that $\mu(\intr{H}) \leq \D(x;\mu) \leq \alpha$. Denote by $H_x=(\intr{H})\compl$ the closed halfspace that satisfies $\mu(H_x\compl)=\mu(\intr{H}) \leq \alpha$. We obtain $\UaFBmu\subseteq H_x$, which is in contradiction with $x \in \intr{\UaFBmu}$.

\textbf{Part~\ref{inclusions second}: $\intr{\Damu} = \intr{\Uamu}$.} Under the considered condition of $\mu$ having contiguous support at $\Damu$ the result follows by Theorem~\ref{theorem:ray basis additional} and part~\ref{contiguous support case} of Theorem~\ref{theorem:ray basis depth regions}. 

\textbf{Part~\ref{inclusions floating body}: $\UaFBmu = \Damu$.} In dimension $d = 1$, the equality of the two expressions is straightforward to verify by rewriting the conditions defining both $\UaFBmu$ and $\Damu$ in terms of the function $F(t) = \mu\left( [t, \infty) \right)$, for $t \in \R$. The general statement follows by projecting $\mu$ via $\pi_u \colon \R^d \to \R \colon y \mapsto \left\langle y, u \right\rangle$, for $u \in \Sph$, into its pushforward measure $\mu_u \in \Meas[\R]$. The image of a halfspace $H_{x,u}$ by $\pi_u$ is the interval $[\left\langle x, u \right\rangle , \infty )$. Knowing the equality of the two concepts of central regions after projecting into $\R$, it is enough to realise that $\Damu = \bigcap_{u \in \Sph} \pi_u^{-1}\left( \D_{\alpha}(\mu_u) \right)$ for $\pi_u^{-1}$ the inverse map to $\pi_u$, and analogously for the floating body $\UaFBmu$. For details see \cite[Theorem~2]{Dyckerhoff2004}.

\textbf{Part~\ref{strict monotonicity}: $\cl{\Uamu} = \Damu$.} Using part~\ref{inclusions second} and the assumption of non-empty interior of both involved sets, \cite[Theorem~1.1.15]{Schneider2014} ensures that $\Damu = \cl{\intr{\Damu}} = \cl{\intr{\Uamu}} = \cl{\Uamu}$ as needed to verify.

\subsection{Proof of Lemma~\ref{lemma:strict monotonicity condition}}

Suppose that $\intr{\Uamu}=\emptyset$. We may apply Theorem~\ref{theorem:ray basis additional} with $F=\Uamu$ and conclude that there exists $H\in\half$ such that $\Uamu \subset \bd{H}$ and $\mu(\intr{H}) \leq \alpha$. At the same time, existence of $x \in \Uamu$ implies $\mu(H) \geq \D(x;\mu) > \alpha$ and consequently $\mu(\bd{H})>0$. Therefore, for any $x\in\Uamu$, $\mu$ is not smooth at $x$.

\subsection{Proof of Theorem~\ref{theorem:covering point in median}}
We distinguish two cases according to whether the median set $\median$ is full-dimensional, or not.

\textbf{Case (i): $\intr{\median}=\emptyset$.}
We write $a = \gamma^*(\mu)$ and start with $U_{a}(\mu)=\bigcup_{\beta>a} D_{\beta}(\mu)$. Since $\intr{D_{\beta}(\mu)}=\emptyset$ for each $\beta>a$ and sets $D_{\beta}(\mu)$ are nested, we conclude that also $\intr{U_{a}(\mu)}=\emptyset$. Choose any $x\in \aff{U_{a}(\mu)}$. From Theorem~\ref{theorem:ray basis additional} applied with $F = U_{a}(\mu)$ it follows that $\R^d$ can be covered by halfspaces $H \in \half(U_a(\mu)) \subset \half(x)$ such that $\mu(\intr{H}) \leq a = \gamma^*(\mu)$ for each such $H$. Since $\aff{\median} \subseteq \aff{U_a(\mu)}$, we obtain that also all points of $\median$ are possible to be used for such a covering, as we needed to show.

\textbf{Case (ii): $\intr{\median}\ne\emptyset$.}
In this case, $\gamma^*(\mu) = \alpha^*(\mu)$. Denote $S_0=\median$. This is a full-dimensional convex set, and for its barycentre $x_0 \in S_0$ and $\nu_0 \in \Meas$ the uniform probability distribution on $S_0$ we know that $\D(x_0;\nu_0) > \e^{-1}$, due to a result of Gr\"unbaum, see \cite[Theorem~3]{Nagy_etal2019}. Let $H_0 \in \half(x_0)$ be a generalized minimizing halfspace of $\mu$ at $x_0$ in the sense of Lemma~\ref{lemma:minimizing halfspace}, i.e. $\mu\left(\intr{H_0}\right) \leq \alpha^*(\mu)$. Writing $\vol(A)$ for the $d$-dimensional volume of a measurable set $A$, this implies $\nu_0(H_0)=\vol(S_0\cap H_0)/\vol(S_0) \geq \D(x_0;\nu_0) > \e^{-1}$. Denote $S_1=\cl{S_0 \setminus H_0}$ with $\vol(S_1)=\vol(S_0)(1-\nu_0(H_0)) < \vol(S_0)(1-\e^{-1})$. We iterate the previous procedure and for $k \geq 1$ define $x_k\in S_k$ to be the barycentre of $S_k$. In each step, we can find a generalized minimizing halfspace $H_k\in\half(x_k)$ of $\mu$ at $x_k$, put $S_{k+1}=\cl{S_k\setminus H_k}$, and again conclude
	\begin{equation*}
	\vol(S_k) < \vol(S_0)\left(1 - \e^{-1}\right)^k.
	\end{equation*}
Necessarily, $\vol(S_k)\rightarrow 0$ as $k\rightarrow \infty$. Denote $S=\bigcap_{k=0}^\infty S_k$. Then $S \subset \median$ is a convex set, $\vol(S)=0$, and consequently $\intr{S}=\emptyset$.

Note that $\median = S_0$ satisfies the assumption \eqref{basic condition} of our main Lemma~\ref{lemma:basic lemma}. Recursively, if $S_k$ satisfies \eqref{basic condition} for some $k=0,1,2,\dots$, then the same holds true for $S_{k+1}$, because
	\begin{equation*}
	S_{k+1}\compl = S_k\compl \cup \intr{H_k}.
	\end{equation*} 
By the induction step, $S_k\compl$ is covered by halfspaces $H \in \half$ with $\mu\left(\intr{H}\right) \leq \alpha^*(\mu)$ and $\intr{H} \cap S_{k+1} \subseteq \intr{H} \cap S_k = \emptyset$. Adding to that collection the halfspace $H_k$ that also satisfies $\mu\left(\intr{H_k}\right) \leq \alpha^*(\mu)$ and $\intr{H_k} \cap S_{k+1} = \emptyset$, we obtain a covering of $S_{k+1}\compl$ with the desired property, i.e. also $S_{k+1}$ satisfies \eqref{basic condition}. We finally show that also $S$ satisfies \eqref{basic condition}. Take any $x \notin S$ and the first index $k=0,1,2,\dots$ such that $x\notin S_k$. There exists $H\in\half$ with $x\in H$, $\mu(\intr{H})\leq \alpha$ and $\intr{H}\cap S_k=\emptyset$. Since $S\subseteq S_k$, it follows that $\intr{H}\cap S=\emptyset$. Lemma~\ref{lemma:basic lemma} can now be applied to the set $S$ to give the desired result, as follows by $\intr{S}=\emptyset$.

\subsection{Proof of Example~\ref{example:not convex}}	\label{section:not convex}

Each point in $\R^2$ is contained in one of the four halfspaces of $\mu$-mass $6$ determined by the dashed lines presented in Figure~\ref{figure:square median}. Therefore $\alpha^*(\mu) \leq 6$. To prove that both points $x$ and $y$ are medians, it is enough to show $\D(x;\mu)=\D(y;\mu)=6$. Because of the symmetry of $\mu$, the problem reduces to showing $\D(y;\mu)=6$. 

\begin{figure}[htpb]
\includegraphics[width=\twofig\textwidth]{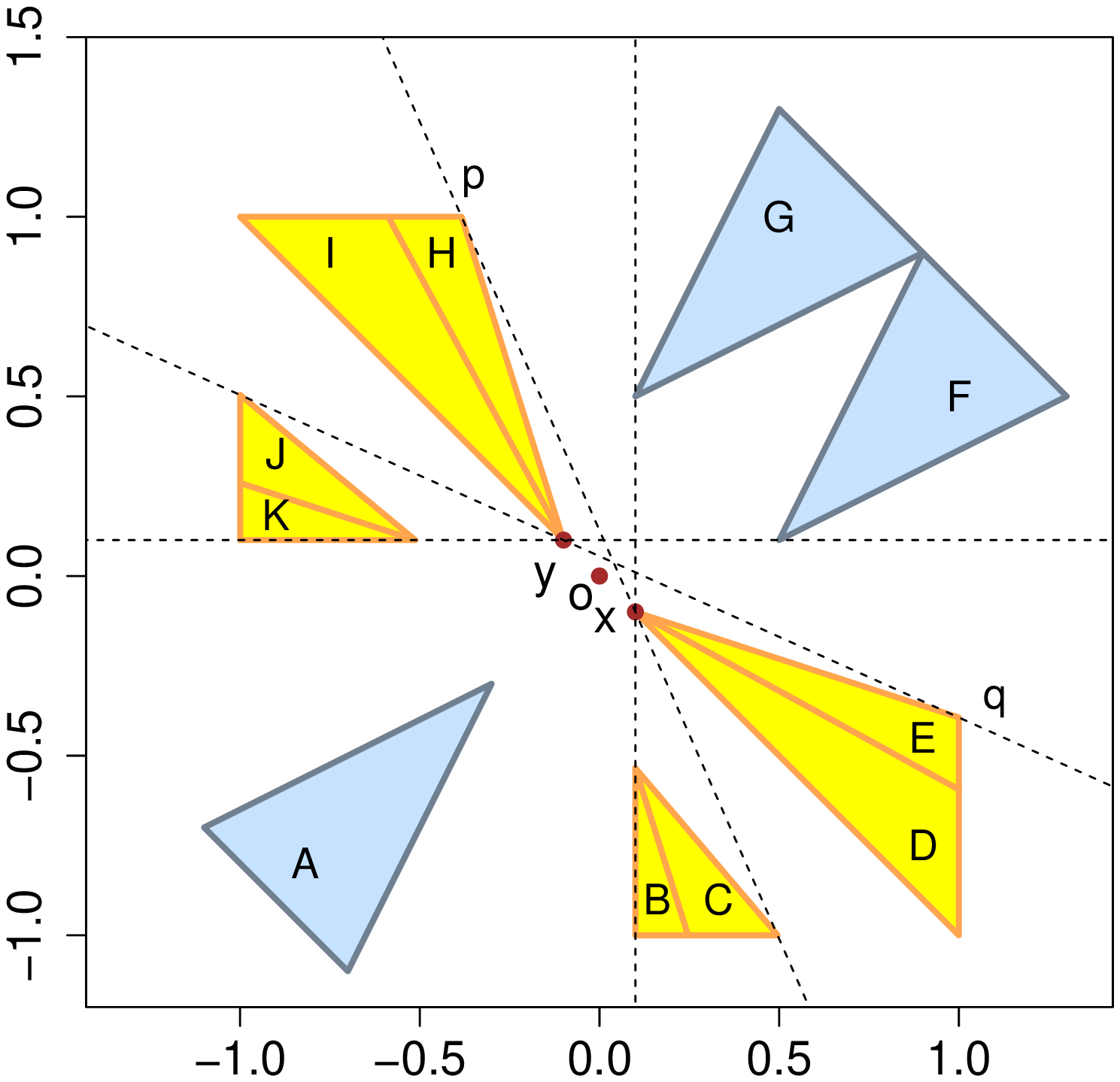} \quad
\includegraphics[width=\twofig\textwidth]{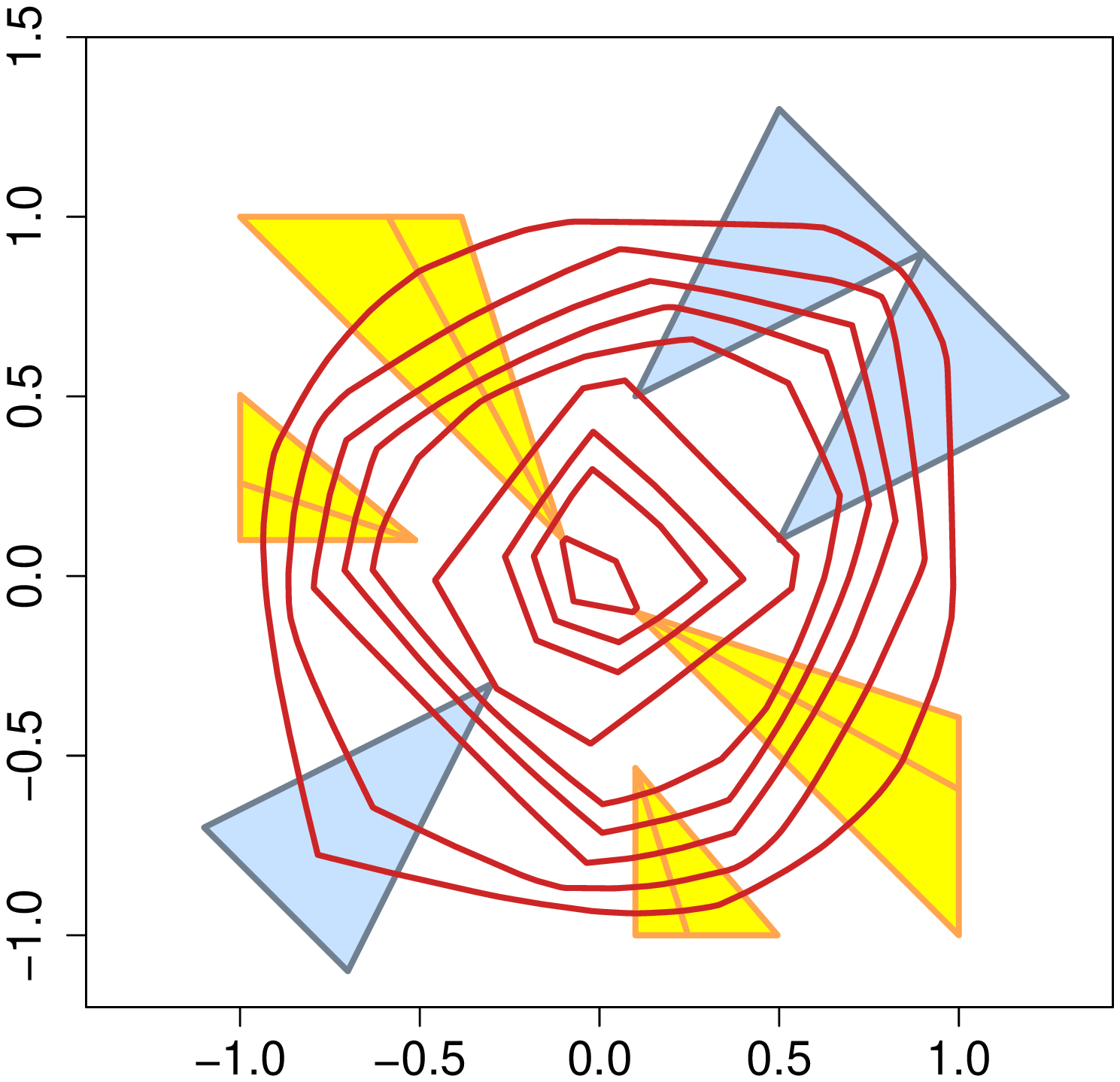}
\caption{Example~\ref{example:not convex}: Each point in $\R^2$ is contained in one of the four halfspaces of $\mu$-mass $6$ determined by the dashed lines (left hand panel). Thus, $\alpha^*(\mu) \leq 6$. On the right hand panel, several numerically computed contours $\Damu$ are displayed as sets with thick coloured boundaries. The median set $\median$ is the innermost convex set in this collection.}
\label{figure:square median}
\end{figure}

Consider an infinite line $l$ passing through $y$ and denote by $H^+$ and $H^-$ the two closed halfspaces determined by $l$. We need to show that neither $\mu(H^+)$ nor $\mu(H^-)$ decreases below $6$ as we rotate $l$ around $y$. Equivalently, as $\mu(\R^2) = 14$, it is enough to show that $6\leq\mu(H^+)\leq 8$ when rotating $l$ by an angle of $\pi$. Several specific positions of $l$ are shown in the left hand panel of Figure~\ref{figure:square min}; the corresponding halfspace $H^+$ for line $l_1$ is the one that does not contain the origin $o$. We start from $l=l_1$ when $\mu(H^+)=6$. As we rotate $l$ around $y$ counter-clockwise, $\mu(H^+)$ remains the same for $l$ being between $l_1$ and $l_2$, since the triangles $A$ and $G$ in Figure~\ref{figure:square min} are symmetric with respect to $y$. Continuing the rotation, $\mu(H^+)$ remains constant and then increases from $l=l_3$ to $l=l_4$, when $\mu(H^+)=7$. For $l$ between $l_4$ and $l_5$, $H^+$ contains the triangles $A$, $B$, $I$, $J$, $K$, so $\mu(H^+)\geq 6$ and $H^-$ contains $D$, $E$, $F$, $G$, meaning that $\mu(H^-)\geq 6$. For $l=l_5$, $\mu(H^+)$ equals $7$ and then decreases until $l=l_6$, when $\mu(H^+)=6$. Afterwards $\mu(H^+)$ increases again and reaches value $8$ for $l=l_7$. Finally, past $l=l_7$ the $\mu$-mass decreases and becomes equal to $6$ when $l$ is horizontal and then increases again and becomes $8$ for $l=-l_1$, where ``$-$" in front of $l_1$ means that the orientation of the halfspace $H^+$ is the opposite from that in the beginning of the rotation. We have shown that $x, y\in \median$ and $\alpha^*(\mu)=6$. The convexity of $\median$ implies that $o\in\median$, as well.

\begin{figure}[htpb]
\includegraphics[width=\twofig\textwidth]{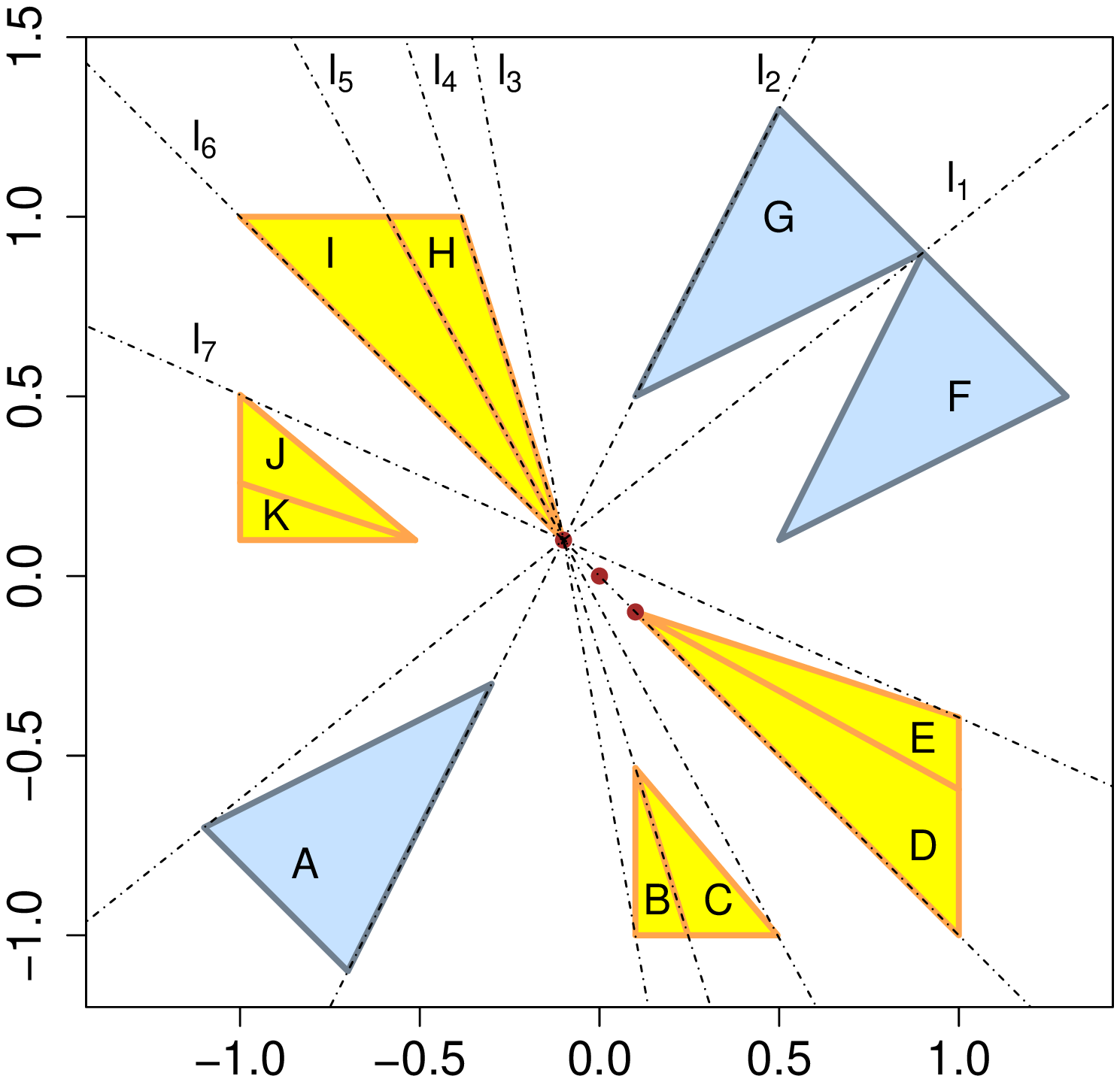}	\quad
\includegraphics[width=\twofig\textwidth]{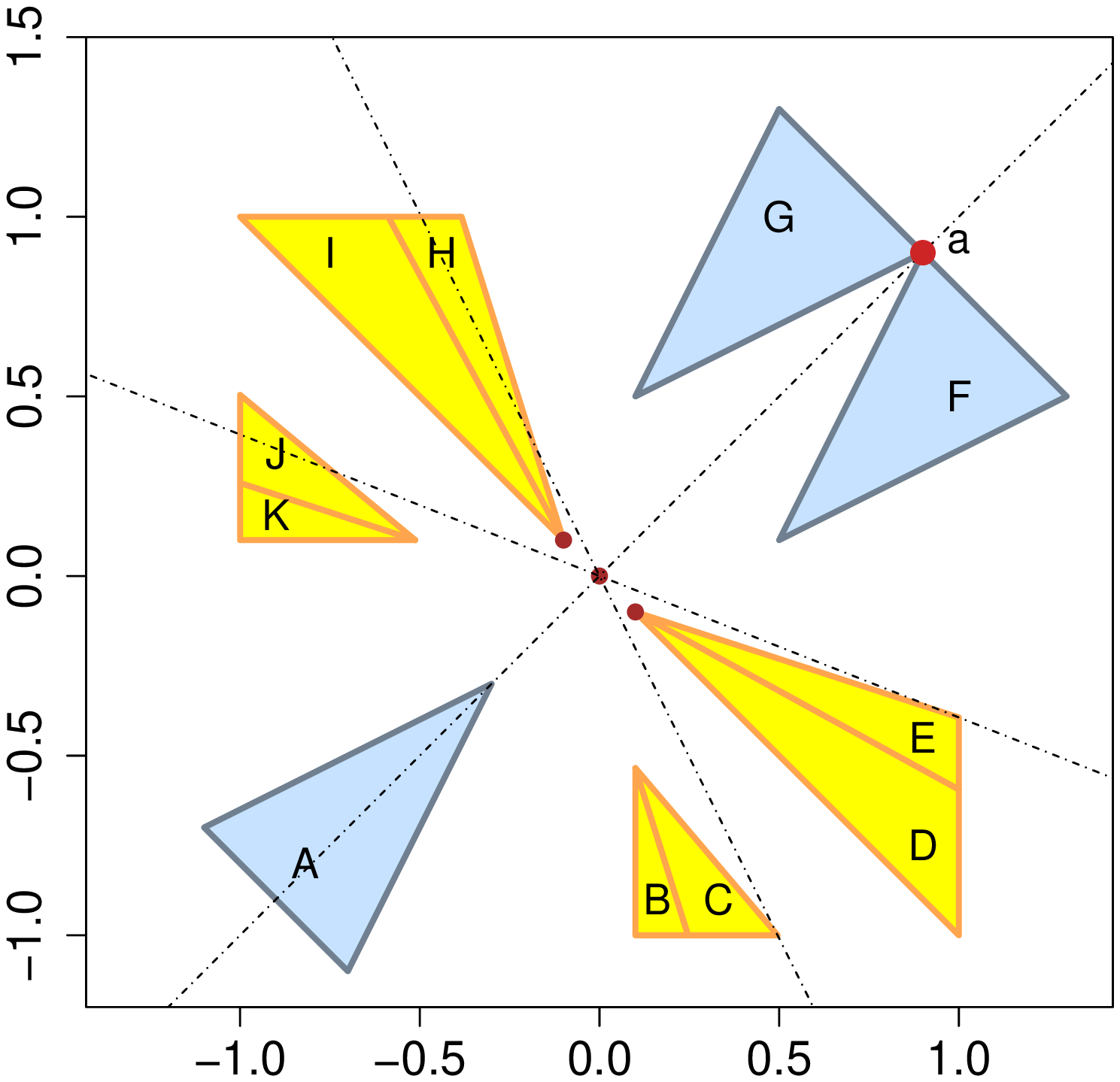}
\caption{Example~\ref{example:not convex}: In the left hand panel we display specific positions of halfspaces when rotating the line $l$ that contains the point $y$. Boundaries of these halfspaces are shown as dashed lines. In the right hand panel we see several halfspaces whose boundary passes through $o$, as described in the proof. Each halfspace from $\half(o)$ that contains the point $a$ has $\mu$-mass at least $6$.}
\label{figure:square min}
\end{figure}

Consider now the covering medians of $\mu$. Certainly $\gamma^*(\mu) = 6$. We are able to cover the whole space by halfspaces whose boundary contains $y$. To see this, consider the halfspaces determined by the horizontal and the vertical line containing $y$, respectively, and the one determined by line $q$ in Figure~\ref{figure:square median}. For $x$, the halfspaces are determined by the line $p$, and the horizontal and vertical line. Therefore, $x, y\in \coverM$.  In the right hand panel of Figure~\ref{figure:square min} we see that each halfspace whose boundary passes through the origin $o$ and contains point $a$ has $\mu$-mass greater than $6$. Therefore, it is not possible to cover $a$ by minimizing halfspaces of $o$, and consequently $o\notin \coverM$.

\subsection{Proof of Theorem~\ref{unique covering median}}
For any $y \in \R^d$ denote by $\half_{\min}(y)$ the collection of those halfspaces $H \in \half(y)$ that satisfy $\mu\left(\intr{H}\right) \leq \gamma^*(\mu)$. In particular, $\half_{\min} = \half_{\min}(x)$. Suppose there is $y\in \coverM$ such that $y\neq x$. Denote $z=(x+y)/2$. Since $\R^d = \bigcup\left\{ H \in \half_{\min}(y) \right\}$, there exists $v \in \Sph$ and $H_{y,v}\in \half_{\min}(y)$ such that $x\in H_{y,v}$. Then $H_{x,v}\subseteq H_{y,v}$ and $\mu(\intr{H_{x,v}})\leq\mu(\intr{H_{y,v}})\leq \gamma^*(\mu)$, meaning that $H_{x,v}\in \half_{\min}(x)$. If $x\in \intr{H_{y,v}}$, then $H_{x,v}\subset H_{z,v}\subset \intr{H_{y,v}}$. Our contiguity assumption then implies $\mu(H_{x,v})<\mu(H_{z,v})\leq \mu(\intr{H_{y,v}})\leq \gamma^*(\mu) \leq \alpha^*(\mu)$, which contradicts $x\in \median$. Thus, $x\in \bd{H_{y,v}}$ and $H_{y,v}\in\half_{\min} (x)$. 

Consider any $w\in \R^d$. If $w\in H_{y,v}$, choose $H_w=H_{y,v}$. Otherwise, the sequence of points $z_n=(1-1/n)z+w/n$, $n=1,2,\dots$ converges to $z$ and $z_n\notin H_{y,v}$. For each $z_n$ there is $H_n\in \half_{\min}(y)$ such that $z_n\in H_n$ and $\mu(\intr{H_n})\leq\gamma^*(\mu)$. Using the same argument as for the halfspace $H_{y,v}$ in the first part of the proof, we conclude that $x\notin \intr{H_n}$, because of $H_n\in \half_{\min} (y)$ and because of our contiguity assumption. Lemma~\ref{lemma:convergent sequence} implies the existence of a convergent subsequence $H_{n_k}\rightarrow H_w \in \half(z)$ such that $\mu(\intr{H_w})\leq \gamma^*(\mu)$. Note that also $y\notin \intr{H_n}$ because $y\in\bd{H_n}$. Using Lemma~\ref{lemma:convergent sequence}, part~\ref{converge to touching}, we conclude that $\{x,y\}\cap \intr{H_w}=\emptyset$. Because $z\in L(x,y)$ and $z\in \bd{H_w}$, it has to be $\{x,y\}\subset \bd{H_w}$, meaning that $H_w\in \half_{\min}(x)\cap \half_{\min}(y)$. Finally, we conclude that for each $w\in \R^d$, there is $H_w\in \half_{\min}(x)\cap \half_{\min}(y)$ such that $w\in H_w$. Then $\half^\prime=\{H_w\colon w\in \R^d\} \subseteq \half_{\min}(x)$ covers $\R^d$ by halfspaces whose interior has mass at most $\gamma^*(\mu)$, that at the same time all contain $y \ne x$, which violates the assumption of our theorem.

\subsection{Proof of Corollary~\ref{Ray basis corollary}}
The first part of the corollary is a direct consequence of Theorem~\ref{theorem:ray basis depth regions}. Because $\mu$ is atomic with finitely many atoms, there are only finitely many unique $\mu$-masses of halfspaces $\half$. Therefore, when applying the Fatou lemma of Lemma~\ref{lemma:convergent sequence} in the proof of Lemma~\ref{lemma:basic lemma}, one obtains a strict inequality $\mu\left(\intr{H(x,F)}\right) < \alpha$.

As for the second claim, first note that $\Damu$ is a convex polytope by \cite[Lemma~1]{Laketa_Nagy2021}. If $\Damu$ is full-dimensional, each face $F$ of $\Damu$ is a subset of a $(d-1)$-dimensional face $\widetilde{F}$ of $\Damu$; if $\dim\left(\Damu\right)<d$, $F \subseteq \widetilde{F} = \Damu$. If we prove our claim for $\widetilde{F}$, it is necessarily true also for $F$. Without loss of generality, we may therefore suppose that $F$ is of dimension $\dim(\widetilde{F}) = \min\left\{\dim(\Damu),d-1\right\}$. Let $H(x,F)$ be the halfspace from the first part of the proof. Denote by $\mu_1 \in \Meas$ the restriction of $\mu$ to $\bd{H(x,F)}$ for $H(x,F) \in \half\left(\Damu, F\right)$ from the first part of the proof, and consider any $x\in F$. If $\D(x;\mu_1)=0$, then there exists a closed $(d-1)$-dimensional halfspace $H_1$ in the hyperplane $\bd{H(x,F)}$ such that $x\in H_1$ and $\mu_1(H_1)=0$. The existence of that minimizing halfspace follows because $\mu$ contains only finitely many atoms, thus there are only finitely many possible values of $\mu_1(H)$ for $H \in \half$. Consider now a slight perturbation of the halfspace $H(x,F)$, in the sense that the unit normal $v \in \Sph$ of $H(x,F)$ is perturbed, but \begin{enumerate*}[label=(\roman*)] \item the $(d-2)$-dimensional affine space $\relbd{H_1}$ remains in the boundary of the perturbed halfspace $H^\prime \in\half(x)$, and \item $H_1 \subset H^\prime$. \end{enumerate*} Because there are only finitely many atoms of $\mu$, it is certainly possible to obtain $H^\prime$ such that $\mu(H^\prime) = \mu(\intr{H(x,F)}) + \mu(H_1) = \mu(\intr{H(x,F)}) < \alpha$, which contradicts $x\in F \subseteq \Damu$. We obtain that $F \subseteq U_{\beta}(\mu_1)$ for $\beta = 0$. Since $U_{0}(\mu_1)$ is, again by the assumption of only finitely many atoms of $\mu$, a polytope whose vertices are atoms of $\mu_1$, $F$ has to be contained in a convex hull of at least $\dim(F)+1$ atoms of $\mu$, all lying in the hyperplane $\bd{H(x,F)}$.

\subsection*{Acknowledgement}
This research was supported by the grant 19-16097Y of the Czech Science Foundation, and by project PRIMUS/17/SCI/3 of Charles University. P.~Laketa was supported by the OP RDE project ``International mobility of research, technical and administrative staff at the Charles University" CZ.02.2.69/0.0/0.0/18\_053/0016976. 


\def\cprime{$'$} \def\polhk#1{\setbox0=\hbox{#1}{\ooalign{\hidewidth
  \lower1.5ex\hbox{`}\hidewidth\crcr\unhbox0}}}

\end{document}